\documentclass[11pt]{article}
\usepackage[margin=1in]{geometry}
\usepackage[utf8]{inputenc}

\usepackage[english]{babel}
\usepackage[T1]{fontenc}
\usepackage{amssymb,amsmath,amstext,amsfonts,amsthm,braket}
\usepackage{mathtools}
\usepackage{verbatim}
\usepackage{relsize}
\usepackage[colorlinks=true, urlcolor=blue, linkcolor=blue, citecolor=blue]{hyperref}
\usepackage{fancyvrb}
\usepackage{tikz-cd}
\usepackage[percent]{overpic}
\usepackage{xcolor}
\usepackage{cleveref}

\usepackage[normalem]{ulem}
\usepackage{blkarray}
\usepackage{subcaption}
\usepackage{enumitem}

\usepackage[authoryear]{natbib}
\usepackage{breakcites}

\usepackage{nicematrix}
\usepackage{afterpage}

\usepackage{algorithm, algorithmic,enumerate,booktabs}

\newcommand{\algorithmicbreak}{\textbf{break}}
\newcommand{\BREAK}{\STATE \algorithmicbreak}

\theoremstyle{plain}
\newtheorem*{theorem*}{Theorem}
\newtheorem{theorem}{Theorem}
\numberwithin{theorem}{section}

\newtheorem{lemma}[theorem]{Lemma}
\newtheorem{corollary}[theorem]{Corollary}
\newtheorem{conjecture}[theorem]{Conjecture}

\theoremstyle{definition}
\newtheorem{definition}[theorem]{Definition}
\newtheorem{remark}[theorem]{Remark}
\newtheorem{example}[theorem]{Example}

\definecolor{MyBlue}{RGB}{0,101,189}
\definecolor{MyRed}{RGB}{234, 114, 55} 
\definecolor{MyGreen}{RGB}{162,173,0}
\definecolor{MyPurple}{RGB}{230,0,200}
\definecolor{MyBrown}{RGB}{165, 42, 42}

\newcommand{\pa}{\text{pa}}
\newcommand{\ch}{\text{ch}}

\newcommand{\jpa}{\text{jpa}}
\newcommand{\im}{\text{Im}}

\newcommand{\cH}{\mathcal{H}}
\newcommand{\cL}{\mathcal{L}}
\newcommand{\cO}{\mathcal{O}}
\newcommand{\od}{\text{od}}
\newcommand{\diag}{\text{diag}}
\newcommand{\flow}{\text{flow}}
\DeclareMathOperator{\PD}{PD}

\newcommand{\footremember}[2]{
    \footnote{#2}
    \newcounter{#1}
    \setcounter{#1}{\value{footnote}}
}

\title{Matching Criterion for Identifiability in Sparse Factor Analysis}
\author{Nils Sturma\footremember{nils}{Technical University of Munich; \href{mailto:nils.sturma@tum.de}{nils.sturma@tum.de}}
\and
Miriam Kranzlmüller\footremember{miriam}{Technical University of Munich; \href{mailto:miriam.kranzlmueller@tum.de}{miriam.kranzlmueller@tum.de}}
\and
Irem Portakal\footremember{irem}{Max Planck Institute for Mathematics in the Sciences, Leipzig; \href{mailto:mail@irem-portakal.de}{mail@irem-portakal.de}}
\and
Mathias Drton\footremember{mathias}{Technical University of Munich, Munich Center for Machine Learning; \href{mailto:mathias.drton@tum.de}{mathias.drton@tum.de}}
}

\date{ }

\begin{document}
\maketitle

\def\spacingset#1{\renewcommand{\baselinestretch}
{#1}\small\normalsize} \spacingset{1}

\spacingset{1.1}

\begin{abstract}
Factor analysis models explain dependence among observed variables by a smaller number  of unobserved factors. A main challenge in confirmatory factor analysis is determining whether the factor loading matrix is identifiable from the observed covariance matrix. The factor loading matrix captures the linear effects of the factors and, if unrestricted, can only be identified up to an orthogonal transformation of the factors. However, in many applications the factor loadings exhibit an interesting sparsity pattern that may lead to identifiability up to column signs. We study this phenomenon by connecting sparse confirmatory factor analysis models to bipartite graphs and providing sufficient graphical conditions for identifiability of the factor loading matrix up to column signs. In contrast to previous work, our main contribution, the matching criterion, exploits sparsity by operating locally on the graph structure, thereby improving existing conditions.  Our criterion is efficiently decidable in time that is polynomial in the size of the graph, when restricting the search steps to sets of bounded size.
\end{abstract}

\section{Introduction}
In factor analysis, a potentially large set of dependent random variables is modeled as a linear combination of a smaller set of underlying latent (unobserved) factors. Factor analysis is ubiquitously applied in fields such as econometrics \citep{fan2008high, amann2016bayesian, aigner1984latent}, psychology \citep{horn1965arationale, reise2000factor, caprara1993thebig, ford1986application, goretzko2023evaluating}, epidemiology \citep{martines1998invited, santos2019principal}, and education \citep{schreiber2006reporting, beavers2013practical}.   It also has applications in causality \citep{pearl2000causality,spirtes:2000} as a building block for models with latent variables \citep{bollen1989structural, barber2022halftrek}. \looseness=-1

Let $X=(X_v)_{v \in V}$ be an observed random vector,  and let $Y=(Y_h)_{h \in \cH}$ be a latent random vector, indexed by finite sets $V$  and $\cH$, respectively. Factor analysis models postulate that each observed variable $X_v$ is a linear function of the \emph{factors} $Y_h$ and noise, that is,
\[
    X = \Lambda Y + \varepsilon,
\]
where $\Lambda = (\lambda_{vh}) \in \mathbb{R}^{|V| \times |\cH|}$ is an unknown coefficient matrix known as the \emph{factor loading matrix}. The elements of  $\varepsilon = (\varepsilon_v)_{v \in V}$  are mutually independent noise variables with mean zero and finite, positive variance.  We consider orthogonal factor analysis, which means that we assume that the latent factors  $(Y_h)_{h \in \cH}$ are mutually independent. The model further assumes that $\varepsilon$ is  independent of $Y$. Without loss of generality, we fix the scale of the factors such that $\text{Var}[Y_h]=1$ for each factor. The main object of study, the covariance matrix of the observed random vector $X$, is then given by
\begin{equation} \label{eq:cov-matrix}
    \Sigma := \text{Cov}[X] = \Lambda \Lambda^{\top} + \Omega,
\end{equation}
where $\Omega$ is a diagonal matrix with entries $\omega_{vv}= \text{Var}[\varepsilon_v]$. 

Our focus is on \emph{confirmatory} factor analysis \citep[Chap.~7]{bollen1989structural}, which pertains to a prespecified model that encodes a scientific hypothesis or was learned previously in an exploratory step. Most interest is typically in models in which the factor loading matrix $\Lambda$ is sparse. In this paper, we assume that the sparsity structure of the factor loading matrix and the number of latent factors $|\cH|$ are fixed and known. Estimation of the factor loadings in confirmatory analyses has been subject to much controversy, due to the difficulties in determining model identifiability \citep{long1983confirmatory}. A factor analysis model is identifiable if the  loading matrix $\Lambda$ can be recovered from the covariance matrix $\Sigma$ in~\eqref{eq:cov-matrix}. If $\Lambda$ is not identifiable, then its estimates are to some degree arbitrary and standard inferential methods invalid~\citep{ximenez2006montecarlo,cox2024weak}. 

In \emph{full} factor analysis, where no restrictions on the factor loading matrix are imposed \citep{drton2007algebraic}, the matrix $\Lambda$ is never identifiable, due to \emph{rotational invariance}. Indeed, for any orthogonal matrix $Q \in \mathbb{R}^{|\cH| \times |\cH|}$, the product $\widetilde{\Lambda} = \Lambda Q$ satisfies
\[
    \widetilde{\Lambda} \widetilde{\Lambda}^{\top} + \Omega = \Lambda Q Q^{\top} \Lambda^{\top} + \Omega = \Lambda \Lambda^{\top} + \Omega
\]
and, thus, $(\widetilde{\Lambda},\Omega)$ determines the same covariance matrix as $(\Lambda,\Omega)$. For this reason, prior work on full factor analysis focuses on identifiability of $\Lambda\Lambda^{\top}$ or, equivalently, of $\Omega$.  \citet{bekker1997generic} characterize generic identifiability, which refers to whether $\Omega$  can be uniquely recovered for almost all parameter choices, except for a few corner cases at the so-called Ledermann bound.
However, for models with sparsity restrictions on  $\Lambda$, the situation may improve, allowing for the identifiability of the loading matrix $\Lambda$ itself up to sign changes of the columns. Identifiability up to column sign is the best we may hope for. If we multiply $\Lambda$ with a diagonal matrix $\Psi$ with entries in $\{\pm 1\}$, then the support of $\Lambda \Psi$ is the same as the support of $\Lambda$ and it still holds that $\Lambda \Psi \Psi^{\top} \Lambda^{\top} = \Lambda \Lambda^{\top}$.

\begin{example} \label{ex:harman}
In a re-analysis of a well-known five-dimensional example of \citet[p.14]{harman1976modern}, \citet[Table 1, Column 3]{trendafilov2017sparse} apply $\ell_1$-penalization techniques and infer the following sparsity pattern in the factor loading matrix:
\[
     \Lambda^{\top} = 
     \begin{pmatrix}
        \lambda_{11} & 0 & \lambda_{31} &\lambda_{41} &\lambda_{51}\\
         0 & \lambda_{22} & 0 & \lambda_{42} 
          & \lambda_{52}
     \end{pmatrix}.
\]
This implies that the observed covariance matrix is given by
\[
    \Sigma = (\sigma_{uv}) = 
    \begin{pmatrix}    \omega_{11}+\lambda_{11}^{2}&0 &\lambda_{11}\lambda_{31}&\lambda_{11}\lambda_{41}&\lambda_{11}\lambda_{51}\\
0 &\omega_{22}+\lambda_{22}^{2}& 0 &\lambda_{22}\lambda_{42}&\lambda_{22}\lambda_{52}\\
\lambda_{11}\lambda_{31}&0&\omega_{33}+\lambda_{31}^{2}&\lambda_{31}\lambda_{41}&\lambda_{31}\lambda_{51}\\
\lambda_{11}\lambda_{41}&\lambda_{22}\lambda_{42}&\lambda_{31}\lambda_{41}&\omega_{44}+\lambda_{41}^{2}+\lambda_{52}^{2}&\lambda_{41}\lambda_{51}+\lambda_{42}\lambda_{52}\\
\lambda_{11}\lambda_{51}&\lambda_{22}\lambda_{52}&\lambda_{31}\lambda_{51}&\lambda_{41}\lambda_{51}+\lambda_{42}\lambda_{52}&\omega_{55}+\lambda_{51}^{2}+\lambda_{52}^{2}
\end{pmatrix}.
\]
For almost every choice of $\Lambda$, we have $\sigma_{34} = \lambda_{31}\lambda_{41} \neq 0$, and the formula
\[
    \sqrt{\frac{\sigma_{13} \sigma_{14} }{\sigma_{34}}} = \sqrt{\frac{\lambda_{11}\lambda_{31} \, \lambda_{11}\lambda_{41}}{\lambda_{31}\lambda_{41}}} = \sqrt{\lambda_{11}^2} =|\lambda_{11} |
\]
shows that we can recover the parameter $\lambda_{11}$ up to sign. Given $|\lambda_{11}|$, the remaining nonzero parameters of the first column of $\Lambda$ are easily found, up to $\text{sign}(\lambda_{11})$. For example, 
\[
     \text{sign}(\lambda_{11}) \lambda_{31}=\frac{\lambda_{11}\lambda_{31}}{|\lambda_{11}|} = \frac{\sigma_{13}}{|\lambda_{11}|},
\]
which is again well-defined for almost all parameter choices. Given $\Lambda_{\star,1}$ up to sign, it is then possible to identify the second column $\Lambda_{\star,2}$ up to $\text{sign}(\lambda_{22})$  using similar formulas. 
\end{example}

\begin{remark}
If  the latent factors are allowed to have arbitrary positive variance instead of fixing $\text{Var}[Y_h]=1$, then we can only hope for identifiability up to column sign and column scaling of $\Lambda$. In this case, the absolute values of the recovered factor loadings within each column can be interpreted as relative strength of effects. 
\end{remark}

The fact that sparsity improves identifiability was noted early in the literature, and there exist many methods in exploratory factor analysis that select a model that is as sparse as possible. \citet{kaiser1958varimax} and \citet{carroll1953analytical} proposed methods that are still used in modern statistical software \citep{sklearn}, optimizing over all rotations such that many factor loadings are close to zero and the remaining loadings have a large absolute value. Developing methods for recovering a sparse factor loading matrix remains a very active field of research. Examples include regularization techniques \citep{ning2011sparse, hirose2012variable, lan2014Sparse, trendafilov2017sparse, scharf2019regularization, goretzko2023regularized, lee2023optimal}, 
rotation methods \citep{liu2023rotation}, correlation thresholding \citep{kim2023structure}, and Bayesian approaches \citep{fruehwirth2025sparse, conti2014bayesian, zhao2016bayesian, rockova2016fastbayesian}.

In this paper, we study identifiability of the factor loading matrix $\Lambda$ from the population covariance matrix $\Sigma = \Lambda \Lambda^{\top} + \Omega$, where the sparsity structure of $\Lambda$ is fixed and known. Reflecting the problem's inherent difficulty, the most prominent sufficient condition for identifiability in confirmatory factor analysis is still the criterion of \citet{anderson1956statistical}, which certifies identifiability of $\Lambda \Lambda^{\top}$. Subsequently, criteria were  developed for identifying $\Lambda$ from $\Lambda \Lambda^{\top}$ up to column sign; see \citet{williams2020identification} or \citet[Section 4]{bai2012statistical}. Examples include the three-indicator rule of \citet{bollen1989structural} and the side-by-side rule of \citet{reilly1996identification}. However,  gaps remain in the existing results. As noted by \citet{hosszejni2026cover}, the model given by the sparse matrix 
\begin{equation} \label{eq:lambda-example}
    \Lambda^{\top} =
    \begin{pmatrix}
        \lambda_{11} & \lambda_{21} & 0 & \lambda_{41} & 0 & 0 \\
        0 & \lambda_{22} & \lambda_{32} & 0 & \lambda_{52} & 0 \\
        0 & 0 & \lambda_{33} & \lambda_{43} & 0 & \lambda_{63} \\
    \end{pmatrix}
\end{equation}
is identifiable up to column sign in an almost sure sense, but the criterion of \citet{anderson1956statistical} and the subsequent developments are not able to certify it.

\tikzset{
      every node/.style={circle, inner sep=0.3mm, minimum size=0.5cm, draw, thick, black, fill=white, text=black},
      every path/.style={thick}
}
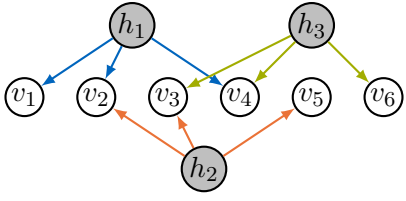
\begin{figure}[t]
\begin{center}
\begin{tikzpicture}[align=center, scale=0.95]
    \node[fill=lightgray] (h1) at (-1,1) {$h_1$};
    \node[fill=lightgray] (h2) at (0,-1) {$h_2$};
    \node[fill=lightgray] (h3) at (1.5,1) {$h_3$};
    
    \node[] (1) at (-2.5,0) {$v_1$};
    \node[] (2) at (-1.5,0) {$v_2$};
    \node[] (3) at (-0.5,0) {$v_3$};
    \node[] (4) at (0.5,0) {$v_4$};
    \node[] (5) at (1.5,0) {$v_5$};
    \node[] (6) at (2.5,0) {$v_6$};
    
    \draw[MyBlue] [-latex] (h1) edge (1);
    \draw[MyBlue] [-latex] (h1) edge (2);
    \draw[MyBlue] [-latex] (h1) edge (4);
    \draw[MyRed] [-latex] (h2) edge (2);
    \draw[MyRed] [-latex] (h2) edge (3);
    \draw[MyRed] [-latex] (h2) edge (5);
    \draw[MyGreen] [-latex] (h3) edge (3);
    \draw[MyGreen] [-latex] (h3) edge (4);
    \draw[MyGreen] [-latex] (h3) edge (6);
\end{tikzpicture}
\caption{Directed graph encoding the sparsity structure in a factor analysis model.}
\label{fig:intro-example}
\end{center}
\end{figure}

In contrast to prior work, we take a graphical perspective to specify the sparsity structure in $\Lambda$ \citep{lauritzen1996graphical, maathuis2019handbook}. For example, the graph in Figure~\ref{fig:intro-example} encodes the sparsity structure in the factor loading matrix given in Equation~\eqref{eq:lambda-example}. When an edge $h \rightarrow v$ is missing in the graph, the corresponding entry $\lambda_{vh}$ is required to be zero. Building on \citet{anderson1956statistical} and \citet{bekker1997generic}, our new \emph{matching criterion} (and an \emph{extension} thereof)  is a purely graphical criterion that exploits sparsity by operating \emph{locally} on the structure of the graph.

Deciding identifiability corresponds to solving the equation system from~\eqref{eq:cov-matrix}. Since the equations are polynomial in the factor loadings $\lambda_{vh}$, identifiability is, in principle, always decidable via Gröbner basis methods from computational algebraic geometry \citep{garcia2010identifying, barber2022halftrek}. But the scope of such methods is limited to small graphs as their complexity  can grow double exponentially with the size of the graph \citep{mayr1997somcecomplexity}. In contrast, 
our new graphical criteria can be checked in polynomial time, provided we restrict a search step to subsets of bounded size.

The organization of the paper is as follows. Section \ref{sec:preliminaries} formally 
introduces the concept of generic sign-identifiability, and we revisit the criteria of \citet{anderson1956statistical} and \citet{bekker1997generic} in Section~\ref{sec:existing-criteria}. Section~\ref{sec:identifiability} presents our main results, the matching criterion and its extension. In Section~\ref{sec:computation}, we show that both criteria are decidable in polynomial time. In Section~\ref{sec:experiments}, we conduct experiments that demonstrate the performance of our criteria and we exemplify how our identifiability criteria are also useful in exploratory factor analysis. The Appendix contains additional results for full factor models (Appendix~\ref{sec:full-factor}),  efficient algorithms  (Appendix~\ref{sec:algos}), all technical proofs (Appendix~\ref{sec:proofs}),  and an explanation of how to decide identifiability using algebraic tools (Appendix~\ref{sec:computational-algebra}). An implementation of the algorithms and code for reproducing the experiments is available at \url{https://github.com/MiriamKranzlmueller/id-factor-analysis}.

\section{Graphical Representation and Identifiability} \label{sec:preliminaries}
Let $G=(V \cup \cH, D)$ be a directed graph, where $V$ and $\cH$ are finite disjoint sets of observed and latent nodes. We assume that the graph $G=(V \cup \cH, D)$ is bipartite, that is, it only contains edges from latent to observed variables such that $D \subseteq \cH \times V$. We refer to such graphs as \emph{factor analysis graphs}. If $G$ contains an edge $(h,v) \in D$, then we also denote this by $h \rightarrow v \in D$. The set $\ch(h)=\{v \in V: h \rightarrow v \in D\}$ contains the children of a latent node $h \in \cH$, and the set $\pa(v)=\{h \in \cH: h \rightarrow v \in D\}$ contains the parents of an observed node $v \in V$. \looseness=-1 

Each bipartite graph defines a factor analysis model, which for our purposes may be identified with a set of covariance matrices. 

\begin{definition}
Let $G=(V \cup \cH, D)$ be a factor analysis graph with $|V|=p$ and $|\cH|=m$, and let $\mathbb{R}^D$ be the space of real $p\times m$ matrices $\Lambda = (\lambda_{vh})$ with support $D$, that is, $\lambda_{vh} = 0$ if $h \rightarrow v \not\in D$. The factor analysis model determined by $G$ is the image $F(G) = \im(\tau_G)$ of the parametrization 
\begin{align*}
\begin{split}
    \tau_G : \mathbb{R}^p_{>0} \times \mathbb{R}^D &\longrightarrow \PD(p)  \\
    (\Omega, \Lambda) &\longmapsto \Omega  + \Lambda \Lambda^{\top},
\end{split}
\end{align*}
where $\PD(p)$ is the cone of positive definite $p \times p$ matrices, and $\mathbb{R}^p_{>0} \subset \PD(p)$ is the subset of diagonal positive definite matrices.
\end{definition}

Identifiability holds if we can recover $\Omega$ and $\Lambda$ from a given matrix $\Sigma \in F(G)$ up to column signs of the matrix $\Lambda$. To make this precise, we write 
$$\mathcal{F}_G(\Omega, \Lambda) = \{(\widetilde{\Omega}, \widetilde{\Lambda}) \in \Theta_G : \tau_G(\widetilde{\Omega}, \widetilde{\Lambda}) = \tau_G(\Omega, \Lambda) \}$$ for the \emph{fiber} of a pair $(\Omega,
\Lambda)$ in the domain $\Theta_{G}=\mathbb{R}^{|V|}_{>0} \times \mathbb{R}^D$ of the parametrization $\tau_G$.

\begin{definition} \label{def:identifiability} 
A factor analysis graph $ G=(V \cup \mathcal{H}, D)$ is said to be \emph{generically sign-identifiable} if
\[
    \mathcal{F}_G(\Omega, \Lambda) = \{(\widetilde{\Omega}, \widetilde{\Lambda}) \in \Theta_G : \widetilde{\Omega} = \Omega \text{ and }  \widetilde{\Lambda}=\Lambda \Psi \text{ for } \Psi \in \{\pm 1\}^{|\cH| \times |\cH|} \text{ diagonal}\}
\]
for almost all $(\Omega, \Lambda) \in \Theta_G$. Moreover, we say that a node $h \in \cH$ in a factor analysis graph $G=(V \cup \mathcal{H}, D)$ is \emph{generically sign-identifiable} if it holds for almost all $(\Omega, \Lambda) \in \Theta_G$ that each parameter pair $(\widetilde{\Omega}, \widetilde{\Lambda})\in \mathcal{F}_G(\Omega, \Lambda)$ satisfies $\widetilde{\Lambda}_{\ch(h),h}=\Lambda_{\ch(h),h}$.
\end{definition}

In Definition~\ref{def:identifiability}, ``almost all'' is meant with respect to the induced Lebesgue measure on $\Theta_G$, considered as an open subset of $\mathbb{R}^{|V|+|D|}$. If a graph is generically sign-identifiable, then for a factor loading matrix $\Lambda$ and a diagonal covariance matrix $\Omega$ drawn randomly from an absolutely continuous distribution, the resulting covariance matrix of the observable vector $X$ will almost surely allow recovery of $\Lambda$ up to column-sign.

\begin{example}
Consider the identification formula for $|\lambda_{11}|$ in Example~\ref{ex:harman} given by
\[
    \sqrt{\frac{\sigma_{13} \sigma_{14} }{\sigma_{34}}} = \sqrt{\frac{\lambda_{11}\lambda_{31} \, \lambda_{11}\lambda_{41}}{\lambda_{31}\lambda_{41}}}.
\]
This formula does not hold if at least one of the parameters $\lambda_{31}$ and $\lambda_{41}$ is equal to zero. Hence for such exceptional parameter pairs $(\Omega,\Lambda)$ we can not establish the correct form of the fiber and identification fails. However, since the set of exceptional pairs forms a Lebesgue measure zero subset of the parameter space, we obtain generic sign-identifiability. 
\end{example}

Note that any node $h$ with $\ch(h)=\emptyset$ is trivially generically sign-identifiable. For later reference, we formally record how generic sign-identifiability of the graph results from generic sign-identifiability of all nodes.

\begin{lemma} \label{lem:nodes-vs-graph-id}
    A factor analysis graph $G=(V \cup \cH, D)$ is generically sign-identifiable if and only if all nodes $h \in \cH$ are generically sign-identifiable.
\end{lemma}

\begin{remark} \label{rem:dimension}
A model can only be generically sign-identifiable if its dimension matches the parameter count $|V|+|D|$. Recently, \citet{drton2025algebraic} proved upper and lower bounds  for the dimension of sparse factor analysis models. The bounds reveal that such models may have dimension strictly smaller  than $|V|+|D|$ and, thus, may be non-identifiable. The bounds also show  that  a necessary condition for a factor analysis graph to be generically sign-identifiable is that each latent node has at least three children. \looseness=-1 
\end{remark}

\section{Existing Criteria} \label{sec:existing-criteria}
Due to rotational indeterminacy, previous work on identifiability of full factor analysis
models focused on identifying the diagonal matrix $\Omega$. If we  require that the upper triangle of the matrix $\Lambda$ is zero, then  existing criteria may also yield generic sign-identifiability. 

\begin{definition} \label{def:ZUTA}
A factor analysis graph $G=(V \cup \cH, D)$ satisfies the \emph{Zero Upper Triangular Assumption (ZUTA)} if there exists an ordering $\prec$ on the latent nodes $\mathcal{H}$ such that
$\ch(h)$ is not contained in $\bigcup_{\ell \succ h} \ch(\ell)$ for all $h \in \cH$.  In this case, we say that $\prec$ is a \emph{ZUTA-ordering} with respect to $G$. 
\end{definition}

ZUTA ensures that the rows and columns of the factor loading matrix $\Lambda$ can be permuted such that the upper triangle of the matrix is zero. Note that ZUTA eliminates rotational indeterminacy. That is, if it holds that $\Sigma - \Omega = \widetilde{\Lambda}  \widetilde{\Lambda}^{\top}$ for a matrix $\widetilde{\Lambda}$ that is zero upper triangular, i.e., $\widetilde{\Lambda}_{ij}=0$ for $i < j$, then it follows from the uniqueness of the Cholesky decomposition that $\widetilde{\Lambda}$ is unique up to column sign, i.e., $\widetilde{\Lambda} = \Lambda \Psi$ for a fixed  matrix $\Lambda$ and a diagonal matrix $\Psi$ with entries in $\{\pm 1\}$.

If a factor analysis graph satisfies the ZUTA condition, then there is an observed node $v_h \in \ch(h)$ for each $h \in \cH$ such that $v_h \in \ch(h)$ and $v_h \not\in \bigcup_{\ell \succ h} \ch(\ell)$.  In particular, it is a necessary condition for ZUTA that there is at least one observed node that only has one latent parent.

\tikzset{
      every node/.style={circle, inner sep=0.3mm, minimum size=0.5cm, draw, thick, black, fill=white, text=black},
      every path/.style={thick}
}
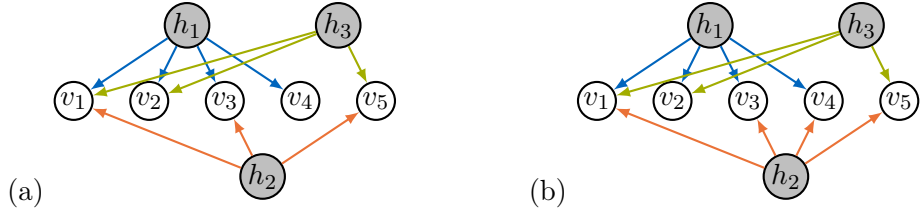
\begin{figure}[t]
\centering
{(a)
\begin{tikzpicture}[align=center]
    \node[fill=lightgray] (h1) at (-1,1) {$h_1$};
    \node[fill=lightgray] (h2) at (0,-1) {$h_2$};
    \node[fill=lightgray] (h3) at (1,1) {$h_3$};
    
    \node[] (1) at (-2.5,0) {$v_1$};
    \node[] (2) at (-1.5,0) {$v_2$};
    \node[] (3) at (-0.5,0) {$v_3$};
    \node[] (4) at (0.5,0) {$v_4$};
    \node[] (5) at (1.5,0) {$v_5$};
    
    \draw[MyBlue] [-latex] (h1) edge (1);
    \draw[MyBlue] [-latex] (h1) edge (2);
    \draw[MyBlue] [-latex] (h1) edge (3);
    \draw[MyBlue] [-latex] (h1) edge (4);
    \draw[MyRed] [-latex] (h2) edge (1);
    \draw[MyRed] [-latex] (h2) edge (3);
    \draw[MyRed] [-latex] (h2) edge (5);
    \draw[MyGreen] [-latex] (h3) edge (1);
    \draw[MyGreen] [-latex] (h3) edge (2);
    \draw[MyGreen] [-latex] (h3) edge (5);
\end{tikzpicture}
\hspace{1.5cm}
(b)
\begin{tikzpicture}[align=center]
   \node[fill=lightgray] (h1) at (-1,1) {$h_1$};
    \node[fill=lightgray] (h2) at (0,-1) {$h_2$};
    \node[fill=lightgray] (h3) at (1,1) {$h_3$};
    
    \node[] (1) at (-2.5,0) {$v_1$};
    \node[] (2) at (-1.5,0) {$v_2$};
    \node[] (3) at (-0.5,0) {$v_3$};
    \node[] (4) at (0.5,0) {$v_4$};
    \node[] (5) at (1.5,0) {$v_5$};
    
    \draw[MyBlue] [-latex] (h1) edge (1);
    \draw[MyBlue] [-latex] (h1) edge (2);
    \draw[MyBlue] [-latex] (h1) edge (3);
    \draw[MyBlue] [-latex] (h1) edge (4);
    \draw[MyRed] [-latex] (h2) edge (1);
    \draw[MyRed] [-latex] (h2) edge (3);
    \draw[MyRed] [-latex] (h2) edge (4);
    \draw[MyRed] [-latex] (h2) edge (5);
    \draw[MyGreen] [-latex] (h3) edge (1);
    \draw[MyGreen] [-latex] (h3) edge (2);
    \draw[MyGreen] [-latex] (h3) edge (5);
\end{tikzpicture}
}
\caption{Two factor analysis graphs. Graph (a) satisfies ZUTA while graph (b) does not. \\}
\label{fig:ZUTA-example}
\end{figure}

\begin{example}
The graph in Figure~\ref{fig:ZUTA-example} (a)  satisfies ZUTA with the ordering $h_1 \prec h_2 \prec h_3$, since $v_4 \in \ch(h_1)$ but $v_4 \not\in \ch(h_2) \cup \ch(h_3)$, and $v_3 \in \ch(h_2)$ but $v_3 \not\in \ch(h_3)$. However, the graph in Figure~\ref{fig:ZUTA-example} (b) does not satisfy ZUTA as no observed node has only one parent.
\end{example}

\begin{remark}
ZUTA is equivalent to the \textit{generalized lower triangular assumption} introduced in \citet{fruehwirth2025sparse}, which operates directly on the matrix $\Lambda$. ZUTA refers to the graph, which is useful to present our graphical criteria in Section~\ref{sec:identifiability}.
\end{remark}

If we consider graphs that satisfy ZUTA, many criteria in the literature directly yield generic sign-identifiability. The most prominent condition for identifiability is still the criterion of \citet{anderson1956statistical}. Since it is originally stated as a pointwise condition, it is also applicable to sparse graphs. To state the result one obtains, we treat the entries of $\Lambda$ as indeterminates and say that a submatrix is generically of rank $k$ if it has rank $k$ for almost all choices of $\Lambda \in \mathbb{R}^D$. Under the assumption that a graph satisfies ZUTA, Theorem 5.1 in \citet{anderson1956statistical} then translates to the following sufficient condition for generic sign-identifiability.

\begin{theorem}[AR-identifiability] \label{thm:anderson-rubin}
Let $G=(V \cup \mathcal{H}, D)$ be a factor analysis graph that satisfies ZUTA.  Then, $G$ is generically sign-identifiable if for any deleted row of $\Lambda = (\lambda_{vh}) \in \mathbb{R}^D$ there remain two disjoint submatrices that are generically of rank $|\cH|$.
\end{theorem}

If generic sign-identifiability can be proven by applying Theorem~\ref{thm:anderson-rubin} for a factor analysis graph, then we say that the graph is \emph{AR-identifiable}. 

\tikzset{
      every node/.style={circle, inner sep=0.3mm, minimum size=0.5cm, draw, thick, black, fill=white, text=black},
      every path/.style={thick}
}
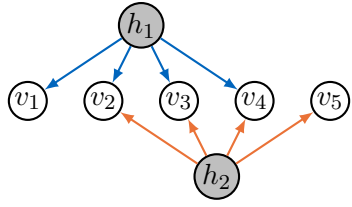
\begin{figure}[t]
\begin{center}
\begin{tikzpicture}[align=center]
    \node[fill=lightgray] (h1) at (-1,1) {$h_1$};
    \node[fill=lightgray] (h2) at (0,-1) {$h_2$};
    
    \node[] (1) at (-2.5,0) {$v_1$};
    \node[] (2) at (-1.5,0) {$v_2$};
    \node[] (3) at (-0.5,0) {$v_3$};
    \node[] (4) at (0.5,0) {$v_4$};
    \node[] (5) at (1.5,0) {$v_5$};
    
    \draw[MyBlue] [-latex] (h1) edge (1);
    \draw[MyBlue] [-latex] (h1) edge (2);
    \draw[MyBlue] [-latex] (h1) edge (3);
    \draw[MyBlue] [-latex] (h1) edge (4);
    \draw[MyRed] [-latex] (h2) edge (2);
    \draw[MyRed] [-latex] (h2) edge (3);
    \draw[MyRed] [-latex] (h2) edge (4);
    \draw[MyRed] [-latex] (h2) edge (5);
\end{tikzpicture}
\caption{AR-identifiable factor analysis graph.}
\label{fig:AR-example}
\end{center}
\end{figure}

\begin{example} \label{ex:AR-identifiable}
The graph in Figure~\ref{fig:AR-example} gives rise to the transpose of $\Lambda \in \mathbb{R}^D$ given by
\[
    \Lambda^{\top} = \begin{pmatrix}
    \lambda_{v_1 h_1} & \lambda_{v_2 h_1} & \lambda_{v_3 h_1} & \lambda_{v_4 h_1} & 0 \\
    0 & \lambda_{v_3 h_2} & \lambda_{v_3 h_2} & \lambda_{v_4 h_2} & \lambda_{v_5 h_2} 
    \end{pmatrix}.
\]
Deleting any row of $\Lambda$ leaves $4$ rows that can always be split in two $2 \times 2$ matrices that generically have rank 2.  Hence, the graph is AR-identifiable.
\end{example}

AR-identifiability requires $|V| \geq 2 |\cH|+1$. For general \emph{full} factor analysis models, \citet{bekker1997generic} solve the problem of generic identifiability (up to orthogonal transformation) in all but certain edge cases. However, the generic nature of their condition implies sign-identifiability results only for dense ZUTA graphs, in which only a permuted upper triangle vanishes.

\begin{definition} \label{def:full-ZUTA-graph}
    A \emph{full-ZUTA graph} is a factor analysis graph $G=(V \cup \cH, D)$ that satisfies ZUTA but contains all other possible edges. That is, there is an ordering $\prec$ on the latent nodes $\cH=\{h_1, \ldots, h_m\}$  such that $h_1 \prec \cdots \prec h_m$, with the property that $\ch(h_1)=V$ and $\ch(h_{i+1})=\ch(h_{i}) \setminus \{v_{i}\}$ for some child  $v_{i} \in \ch(h_{i})$.
\end{definition}

\tikzset{
      every node/.style={circle, inner sep=0.3mm, minimum size=0.5cm, draw, thick, black, fill=white, text=black},
      every path/.style={thick}
}
\begin{figure}[t]
\centering
{(a)
\begin{tikzpicture}[align=center]
    \node[fill=lightgray] (h1) at (-1,1) {$h_1$};
    \node[fill=lightgray] (h2) at (0.5,-1) {$h_2$};
    \node[fill=lightgray] (h3) at (1,1) {$h_3$};
    
    \node[] (1) at (-2.5,0) {$v_1$};
    \node[] (2) at (-1.5,0) {$v_2$};
    \node[] (3) at (-0.5,0) {$v_3$};
    \node[] (4) at (0.5,0) {$v_4$};
    \node[] (5) at (1.5,0) {$v_5$};
    \node[] (6) at (2.5,0) {$v_6$};
    
    \draw[MyBlue] [-latex] (h1) edge (1);
    \draw[MyBlue] [-latex] (h1) edge (2);
    \draw[MyBlue] [-latex] (h1) edge (3);
    \draw[MyBlue] [-latex] (h1) edge (4);
    \draw[MyBlue] [-latex] (h1) edge (5);
    \draw[MyBlue] [-latex] (h1) edge (6);
    \draw[MyRed] [-latex] (h2) edge (2);
    \draw[MyRed] [-latex] (h2) edge (3);
    \draw[MyRed] [-latex] (h2) edge (4);
    \draw[MyRed] [-latex] (h2) edge (5);
    \draw[MyRed] [-latex] (h2) edge (6);
    \draw[MyGreen] [-latex] (h3) edge (3);
    \draw[MyGreen] [-latex] (h3) edge (4);
    \draw[MyGreen] [-latex] (h3) edge (5);
    \draw[MyGreen] [-latex] (h3) edge (6);
\end{tikzpicture}
\hspace{0.5cm}
(b)
\begin{tikzpicture}[align=center]
    \node[fill=lightgray] (h1) at (-1,1) {$h_1$};
    \node[fill=lightgray] (h2) at (0.5,-1) {$h_2$};
    \node[fill=lightgray] (h3) at (1,1) {$h_3$};
    
    \node[] (1) at (-2.5,0) {$v_1$};
    \node[] (2) at (-1.5,0) {$v_2$};
    \node[] (3) at (-0.5,0) {$v_3$};
    \node[] (4) at (0.5,0) {$v_4$};
    \node[] (5) at (1.5,0) {$v_5$};
    \node[] (6) at (2.5,0) {$v_6$};
    \node[] (7) at (3.5,0) {$v_7$};
    
    \draw[MyBlue] [-latex] (h1) edge (1);
    \draw[MyBlue] [-latex] (h1) edge (2);
    \draw[MyBlue] [-latex] (h1) edge (3);
    \draw[MyBlue] [-latex] (h1) edge (4);
    \draw[MyBlue] [-latex] (h1) edge (5);
    \draw[MyBlue] [-latex] (h1) edge (6);
    \draw[MyBlue] [-latex] (h1) edge (7);
    \draw[MyRed] [-latex] (h2) edge (2);
    \draw[MyRed] [-latex] (h2) edge (3);
    \draw[MyRed] [-latex] (h2) edge (4);
    \draw[MyRed] [-latex] (h2) edge (5);
    \draw[MyRed] [-latex] (h2) edge (6);
    \draw[MyRed] [-latex] (h2) edge (7);
    \draw[MyGreen] [-latex] (h3) edge (3);
    \draw[MyGreen] [-latex] (h3) edge (4);
    \draw[MyGreen] [-latex] (h3) edge (5);
    \draw[MyGreen] [-latex] (h3) edge (6);
    \draw[MyGreen] [-latex] (h3) edge (7);
\end{tikzpicture}
}
\caption{Two full-ZUTA graphs. \\}
\label{fig:full-ZUTA}
\end{figure}
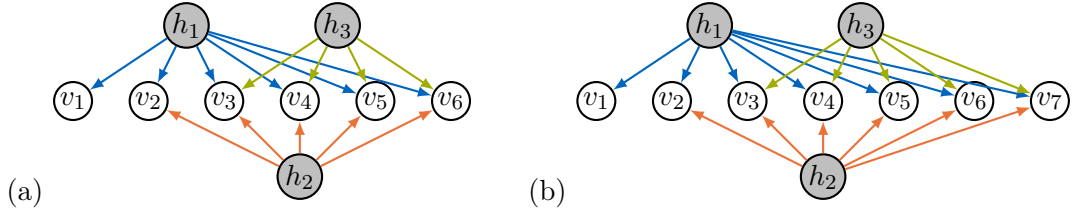

As an example, Figure~\ref{fig:full-ZUTA} (a) displays the full-ZUTA graph on 3 latent and 6 observed nodes. For full-ZUTA graphs, the criterion from \citet{bekker1997generic} directly translates into the following sufficient condition for generic sign-identifiability.

\begin{theorem}[BB-identifiability] \label{thm:bekker-berge}
Let $G=(V \cup \mathcal{H}, D)$ be a full-ZUTA graph. Then, $G$ is generically  sign-identifiable if $|V| + |D| < \binom{|V|+1}{2}$.
\end{theorem}

If a full-ZUTA graph is generic sign-identifiability by Theorem~\ref{thm:bekker-berge}, then we term the graph \emph{BB-identifiable}. Note that $|V| + |D| = |V|(|\cH|+1) - \binom{|\cH|}{2}$ in a full-ZUTA graph. If $|V| + |D| > \binom{|V|+1}{2}$, then the parameter count is larger than the dimension of the ambient space of symmetric matrices, and full-ZUTA graphs are not generically sign-identifiable; recall Remark~\ref{rem:dimension}. Hence, the only remaining open cases where identifiability of full-ZUTA graphs is unknown are models ``at the Ledermann bound'' where $|V| + |D| = \binom{|V|+1}{2}$. 

\begin{example}
    Figure~\ref{fig:full-ZUTA} shows two full-ZUTA graphs. Graph (b) is BB-identifiable because
    $
        |V| + |D| = 24 \, < \, 28 = \binom{7+1}{2}.
    $
    Graph (a), on the other hand, has
    $
        |V| + |D| = \, 21 \, = \binom{6+1}{2}.
    $
    As already noted by \citet{wilson1939theresolution}, the fiber for graph (a), with $|V|=6$ and $|\cH|=3$, generically contains two diagonal matrices and two corresponding factor loading matrices together with their symmetries given by the sign changes of the columns.
\end{example}

\tikzset{
      every node/.style={circle, inner sep=0.3mm, minimum size=0.5cm, draw, thick, black, fill=white, text=black},
      every path/.style={thick}
}
\begin{figure}[t]
\begin{center}
(a)
\begin{tikzpicture}[align=center]
    \node[fill=lightgray] (h1) at (-0.5,1) {$h_1$};
    \node[fill=lightgray] (h2) at (0,-1) {$h_2$};
    
    \node[] (1) at (-2.5,0) {$v_1$};
    \node[] (2) at (-1.5,0) {$v_2$};
    \node[] (3) at (-0.5,0) {$v_3$};
    \node[] (4) at (0.5,0) {$v_4$};
    \node[] (5) at (1.5,0) {$v_5$};
    
    \draw[MyBlue] [-latex] (h1) edge (1);
    \draw[MyBlue] [-latex] (h1) edge (2);
    \draw[MyBlue] [-latex] (h1) edge (3);
    \draw[MyBlue] [-latex] (h1) edge (4);
    \draw[MyBlue] [-latex] (h1) edge (5);
    \draw[MyRed] [-latex] (h2) edge (2);
    \draw[MyRed] [-latex] (h2) edge (3);
    \draw[MyRed] [-latex] (h2) edge (4);
    \draw[MyRed] [-latex] (h2) edge (5);
\end{tikzpicture}
\hspace{2cm}
(b)
\begin{tikzpicture}[align=center]
    \node[fill=lightgray] (h1) at (-0.5,1) {$h_1$};
    \node[fill=lightgray] (h2) at (1,-1) {$h_2$};
    
    \node[] (1) at (-2.5,0) {$v_1$};
    \node[] (2) at (-1.5,0) {$v_2$};
    \node[] (3) at (-0.5,0) {$v_3$};
    \node[] (4) at (0.5,0) {$v_4$};
    \node[] (5) at (1.5,0) {$v_5$};
    
    \draw[MyBlue] [-latex] (h1) edge (1);
    \draw[MyBlue] [-latex] (h1) edge (2);
    \draw[MyBlue] [-latex] (h1) edge (3);
    \draw[MyBlue] [-latex] (h1) edge (4);
    \draw[MyBlue] [-latex] (h1) edge (5);
    \draw[MyRed] [-latex] (h2) edge (4);
    \draw[MyRed] [-latex] (h2) edge (5);
\end{tikzpicture}
\caption{Full-ZUTA graph and a sparse subgraph.}
\label{fig:non-generic-points}
\end{center}
\end{figure}
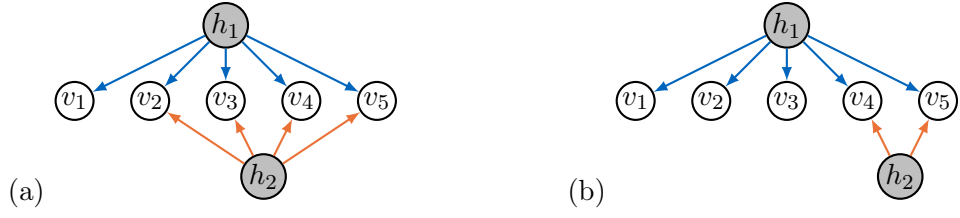

\begin{remark}
    Generic sign-identifiability of full-ZUTA graphs does not imply identifiability of sparse subgraphs, since the models corresponding to subgraphs might be non-generic points in the model given by the full-ZUTA graph. For example, consider the full-ZUTA graph in Figure~\ref{fig:non-generic-points} (a) that is generically sign-identifiable by Theorem~\ref{thm:bekker-berge}. The graph in Figure~\ref{fig:non-generic-points} (b) is a sparse subgraph. Since in this graph $|\ch(h_2)| < 3$, it follows that the model has not expected dimension and is hence not generically sign-identifiable; recall Remark~\ref{rem:dimension}.
\end{remark}

 The following example shows two graphs, that are generically sign-identifiable but no known general criterion is able to certify it.

\begin{example}
The loading matrix for the graph in  Figure~\ref{fig:not-AR-identifiable} has transpose
\[
    \Lambda^{\top} = \begin{pmatrix}
    \lambda_{v_1 h_1} & \lambda_{v_2 h_1} & \lambda_{v_3 h_1} & \lambda_{v_4 h_1} & \lambda_{v_5 h_1} & \lambda_{v_6 h_1} & 0 & 0 & 0 \\
    0 & \lambda_{v_2 h_2} & \lambda_{v_3 h_2} & \lambda_{v_4 h_2} & \lambda_{v_5 h_2} & \lambda_{v_6 h_2} & 0 & 0 & 0 \\
    0 & 0 & \lambda_{v_3 h_3} & \lambda_{v_4 h_3} & \lambda_{v_5 h_3} & 0 & 0 & 0 & 0 \\
    0 & 0 & 0 & \lambda_{v_4 h_4} & 0 & \lambda_{v_6 h_4} & \lambda_{v_7 h_4} & \lambda_{v_8 h_4} & \lambda_{v_9 h_4}
\end{pmatrix}.
\]
The graph is not BB-identifiable as it is not  full-ZUTA. To see that it is not AR-identifiable, delete the row of $\Lambda$ indexed by $v_4$. If we form two $4 \times 4$-matrices out of the remaining $8$ rows, then one of these matrices has to contain at least two rows indexed by $v_7$, $v_8$ or $v_9$. This matrix has at most rank three, which disproves AR-identifiability. Another example that is neither AR- nor BB-identifiable is the graph in Figure~\ref{fig:intro-example}. Using the criteria we develop in the next section, we can certify identifiability of both graphs; also see Example~\ref{ex:M-identifibility}. 
\end{example}

\tikzset{
      every node/.style={circle, inner sep=0.3mm, minimum size=0.5cm, draw, thick, black, fill=white, text=black},
      every path/.style={thick}
}
\begin{figure}[t]
\begin{center}
\begin{tikzpicture}[align=center, scale=0.95]
    \node[fill=lightgray] (h1) at (-1,1) {$h_1$};
    \node[fill=lightgray] (h2) at (0,-1) {$h_2$};
    \node[fill=lightgray] (h3) at (1,1) {$h_3$};
    \node[fill=lightgray] (h4) at (3,1) {$h_4$};
    
    \node[] (1) at (-2.5,0) {$v_1$};
    \node[] (2) at (-1.5,0) {$v_2$};
    \node[] (3) at (-0.5,0) {$v_3$};
    \node[] (4) at (0.5,0) {$v_4$};
    \node[] (5) at (1.5,0) {$v_5$};
    \node[] (6) at (2.5,0) {$v_6$};
    \node[] (7) at (3.5,0) {$v_7$};
    \node[] (8) at (4.5,0) {$v_8$};
    \node[] (9) at (5.5,0) {$v_9$};
    
    \draw[MyBlue] [-latex] (h1) edge (1);
    \draw[MyBlue] [-latex] (h1) edge (2);
    \draw[MyBlue] [-latex] (h1) edge (3);
    \draw[MyBlue] [-latex] (h1) edge (4);
    \draw[MyBlue] [-latex] (h1) edge (5);
    \draw[MyBlue] [-latex] (h1) edge (6);
    \draw[MyRed] [-latex] (h2) edge (2);
    \draw[MyRed] [-latex] (h2) edge (3);
    \draw[MyRed] [-latex] (h2) edge (4);
    \draw[MyRed] [-latex] (h2) edge (5);
    \draw[MyRed] [-latex] (h2) edge (6);
    \draw[MyGreen] [-latex] (h3) edge (3);
    \draw[MyGreen] [-latex] (h3) edge (4);
    \draw[MyGreen] [-latex] (h3) edge (5);
    \draw[MyBrown] [-latex] (h4) edge (4);
    \draw[MyBrown] [-latex] (h4) edge (6);
    \draw[MyBrown] [-latex] (h4) edge (7);
    \draw[MyBrown] [-latex] (h4) edge (8);
    \draw[MyBrown] [-latex] (h4) edge (9);
\end{tikzpicture}
\caption{Sparse factor analysis graphs that is not AR-identifiable nor BB-identifiable.}
\label{fig:not-AR-identifiable}
\end{center}
\end{figure}
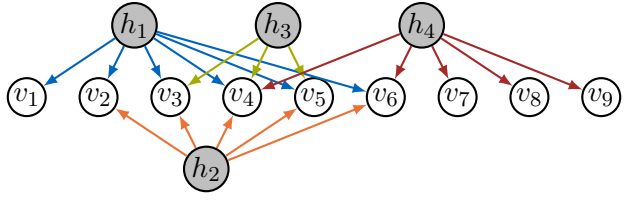

Finally, we note that BB-identifiability subsumes AR-identifiability for full-ZUTA graphs.

\begin{corollary} \label{cor:BB-subsumes-AR}
    Let $G=(V \cup \cH, D)$ be a full-ZUTA graph with $|\cH|\geq 2$ latent nodes that is AR-identifiable. Then, $G$ is also BB-identifiable.
\end{corollary}

However, there are full-ZUTA graphs that are BB- but not AR-identifiable. The smallest example has $|V|=8$ observed nodes and $|\cH|=4$ latent nodes.

\section{Main Identifiability Results} \label{sec:identifiability}
In this section, we derive novel graphical criteria that are sufficient for generic sign-identifiability in sparse factor analysis graphs. As we will show, the criteria strictly generalize AR- and BB-identifiability for ZUTA graphs, and are capable of certifying identifiability of models not covered by the AR- nor BB-criterion. 

\subsection{Matching Criterion}

Our first criterion takes the form of a recursive procedure and is based on a graphical extension of the Anderson-Rubin criterion that can be applied locally at a given node.
In the AR criterion, for each observed node $v \in V$, we need to find disjoint sets $U,W \subseteq V\setminus\{v\}$ with $|U|=|W|=|\cH|$ such that $\det(\Lambda_{U,\cH}) \neq 0$ and $\det(\Lambda_{W,\cH}) \neq 0$. This is equivalent to $\det([\Lambda \Lambda^{\top}]_{U,W}) \neq 0$.  Our main idea is to derive, and locally apply, a modified version of the AR criterion that also considers sets $U,W$ with cardinality smaller than $|\cH|$. In doing so, we need to ensure that $\det([\Lambda \Lambda^{\top}]_{U,W}) \neq 0$, i.e., we need to characterize when minors of $\Lambda \Lambda^{\top}$ vanish. This can be achieved via the concept of trek-separation \citep{sullivant2010trek} and leads to the following definition. \looseness=-1

\begin{definition}
Let $G=(V \cup \mathcal{H}, D)$ be a factor analysis graph, 
and let $A,B\subseteq V$ be two subsets of equal cardinality, $|A|=|B|=n$.  A \emph{matching} of $A$ and $B$ is a system $\Pi = \{\pi_1, \ldots, \pi_n\}$ consisting of paths of the form
\[
    \pi_i: v_i \leftarrow h_i \rightarrow w_i, \quad i=1,\dots,n,
\]
where all $h_i \in \mathcal{H}$, and $\{v_1, \ldots, v_n\}=A$ and $\{w_1, \ldots, w_n\}=B$. A matching is \emph{intersection-free} if the $h_i$ are all distinct, and a matching \emph{avoids} $\mathcal{L} \subseteq \mathcal{H}$ if $\mathcal{L} \cap \{h_1, \ldots, h_n\} = \emptyset$. 
\end{definition}

\begin{example}
Consider the sets $A=\{v_2,v_3\}$ and $B=\{v_4,v_5\}$ in the graph from Figure~\ref{fig:non-generic-points} (a). The system
$
    \{v_2 \leftarrow h_1 \rightarrow v_3, v_4 \leftarrow h_2 \rightarrow v_5\}
$
is an intersection-free matching of $A$ and $B$. If instead $A=\{v_1,v_2,v_3\}$ and $B=\{v_1,v_4,v_5\}$, then any matching between $A$ and $B$ has an intersection. An example is given by the set of paths
$
    \{v_1 \leftarrow h_1 \rightarrow v_1, v_2 \leftarrow h_1 \rightarrow v_3, v_4 \leftarrow h_2 \rightarrow v_5\}
$
that intersects in the latent node $h_1$.
\end{example}

Our main tool is a lemma that considers determinants of submatrices of $\Lambda \Lambda^{\top}$ for $\Lambda \in \mathbb{R}^D$. Here, we view the determinant as a polynomial in the indeterminates $\lambda_{vh}$, that is, we view it as an algebraic object without reference to its evaluation at specific values. 

\begin{lemma} \label{lem:determinant}
Let $G=(V \cup \mathcal{H}, D)$ be a factor analysis graph, and let $\Lambda \in \mathbb{R}^D$. For two subsets $A, B \subseteq V$  of equal cardinality, $\det( [\Lambda \Lambda^{\top}]_{A,B} )$ is not the zero polynomial
if and only if there is an intersection-free matching of $A$ and $B$.
\end{lemma}

Applying Lemma~\ref{lem:determinant} to Anderson and Rubin's theorem yields the following corollary.

\begin{corollary} \label{cor:ar-equivalence}
    Let $G=(V \cup \mathcal{H}, D)$ be a factor analysis graph that satisfies ZUTA.  Then $G$ is AR-identifiable if and only if for all $v \in V$, there exist two disjoint sets $W,U \subseteq V \setminus \{v\}$ with $|W|=|U|=|\cH|$ such that there is an intersection-free matching between $W$ and $U$. 
\end{corollary}

\begin{example}
We saw in Example~\ref{ex:AR-identifiable} that the graph in Figure~\ref{fig:AR-example} is AR-identifiable. Corollary~\ref{cor:ar-equivalence} allows us to certify AR-identifiability in a fully graphical way without relating to the factor loading matrix. For example, for node $v_5$, we observe that the two sets $U=\{v_1,v_2\}$ and $W=\{v_3,v_4\}$ have intersection-free matching
$
    \{v_1 \leftarrow h_1 \rightarrow v_2, v_3 \leftarrow h_2 \rightarrow v_4\}.
$
\end{example}

\begin{remark} \label{rem:comparison-matchings}
    The use of matchings to verify AR-identifiability also appears in recent work of \citet[Proposition 2]{hosszejni2026cover} who make a connection between computing {classical} maximal matchings in bipartite graphs and verifying AR-identifiability. They consider matchings that are defined on \emph{duplicate bipartite} graphs, which are constructed by first duplicating all latent nodes of the original graph and then duplicating the edges connecting these new latent nodes to the original observed nodes. The criterion of \citet{hosszejni2026cover} then establishes AR-identifiability by checking whether the duplicate bipartite graph admits a maximal matching that covers all latent nodes, both the original and their duplicates. However, this approach is not feasible when we modify Corollary~\ref{cor:ar-equivalence} to be locally applicable, as we do next. The reason is that if not all latent nodes are part of the matching, we do not know a priori which nodes we should consider in the bipartite graph. Therefore, we consider intersection-free matchings defined with respect to the original factor analysis graph. 
\end{remark}

We are now ready to define our new matching criterion, which 
operates ``node-wise'' and considers generic sign-identifiability for individual latent nodes $h \in \cH$.

\begin{definition} \label{def:matching-criterion}
    Fix a latent node $h \in \cH$ in factor analysis graph $G=(V \cup \mathcal{H}, D)$. A tuple $(v, W, U, S) \in V \times 2^{V} \times 2^{V} \times 2^{ \cH \setminus \{h\}}$ satisfies the \emph{matching criterion} with respect to $h$ if
    \begin{itemize}
        \item[(i)] $\pa(v)\setminus S = \{h\}$ and $v \not\in W \cup U$,
        \item[(ii)] $W$ and $U$ are disjoint, nonempty sets of equal cardinality, 
        \item[(iii)] there exists an intersection-free matching of $W$ and $U$ that avoids $S$,
        \item[(iv)] there is no intersection-free matching of $\{v\} \cup W$ and $\{v\} \cup U$ that avoids $S$.
    \end{itemize}
\end{definition}

If $(v, W, U, S)$ satisfies the matching criterion with respect to $h$, then Condition (iii) ensures $\det([\Lambda \Lambda^{\top}]_{W,U}) \neq 0$, and Condition (iv) ensures  $\det([\Lambda \Lambda^{\top}]_{\{v\} \cup W,\{v\} \cup U}) = 0$ after removing the nodes in $S$ from the graph. The Laplace expansion of determinants then allows us to find a rational formula for $\lambda_{vh}^2$ in terms of the entries of the covariance matrix. We can thus identify $\lambda_{vh}$ up to sign. Having identified  parameter $\lambda_{v h}$ for one child $v \in \ch(h)$, it is easy to certify sign-identifiability of $h$, i.e., to identify the remaining parameters $\lambda_{uh}$ for  $u \in \ch(h) \setminus \{v\}$ up to the same sign. This is formalized in our first main result.

\begin{theorem}[M-identifiability] \label{thm:identifiability}
Let $G=(V \cup \mathcal{H}, D)$ be a factor analysis graph, and fix a latent node $h \in \cH$.  Suppose that the tuple $(v, W, U, S) \in V \times 2^{V} \times 2^{V} \times 2^{ \cH \setminus \{h\}}$ satisfies the matching criterion with respect to $h$. If all nodes $\ell \in S$ are generically sign-identifiable, then $h$ is generically sign-identifiable.
\end{theorem}

Theorem~\ref{thm:identifiability} provides a way to recursively certify generic sign-identifiability of a factor analysis graph by checking whether all nodes $h \in \cH$ are generically sign-identifiable; recall Lemma~\ref{lem:nodes-vs-graph-id}. If generic sign-identifiability can be certified recursively by Theorem~\ref{thm:identifiability}, then we call the factor analysis graph \emph{M-identifiable}. The details of an efficient algorithm to check M-identifiability using max-flow techniques are described in Appendix~\ref{sec:algos}.

\tikzset{
      every node/.style={circle, inner sep=0.3mm, minimum size=0.5cm, draw, thick, black, fill=white, text=black},
      every path/.style={thick}
}
\begin{figure}[t]
\centering
{
\begin{tikzpicture}[align=center]
    \node[fill=lightgray] (h1) at (-1,1) {$h_1$};
    \node[fill=lightgray] (h2) at (0.5,-1) {$h_2$};
    \node[fill=lightgray] (h3) at (0.5,1) {$h_3$};
    
    \node[] (1) at (-2.5,0) {$v_1$};
    \node[] (2) at (-1.5,0) {$v_2$};
    \node[] (3) at (-0.5,0) {$v_3$};
    \node[] (4) at (0.5,0) {$v_4$};
    \node[] (5) at (1.5,0) {$v_5$};
    \node[] (6) at (2.5,0) {$v_6$};
    
    \draw[MyBlue] [-latex] (h1) edge (1);
    \draw[MyBlue] [-latex] (h1) edge (2);
    \draw[MyBlue] [-latex] (h1) edge (3);
    \draw[MyBlue] [-latex] (h1) edge (4);
    \draw[MyBlue] [-latex] (h1) edge (5);
    \draw[MyBlue] [-latex] (h1) edge (6);
    \draw[MyRed] [-latex] (h2) edge (2);
    \draw[MyRed] [-latex] (h2) edge (3);
    \draw[MyRed] [-latex] (h2) edge (4);
    \draw[MyRed] [-latex] (h2) edge (5);
    \draw[MyRed] [-latex] (h2) edge (6);
    \draw[MyGreen] [-latex] (h3) edge (3);
    \draw[MyGreen] [-latex] (h3) edge (4);
    \draw[MyGreen] [-latex] (h3) edge (5);
\end{tikzpicture}
}
\caption{M-identifiable sparse factor analysis graph.}
\label{fig:M-example}
\end{figure}
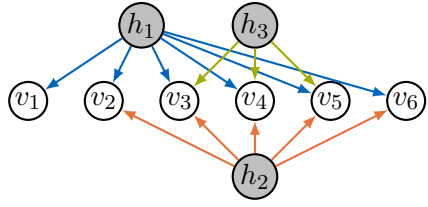

\begin{example} \label{ex:M-identifibility}
The factor analysis graph in Figure~\ref{fig:M-example} is not AR-identifiable since $|V|=2|\cH|$. However, it is M-identifiable. We recursively check all latent nodes $\cH = \{ h_1, h_2, h_3 \}$.
\begin{itemize}
\item[\underline{$h_1$:}] Take $v=v_1$, $S=\emptyset$, $U=\{v_2,v_6\}$, $W=\{v_3,v_4\}$. Conditions (i) and (ii) are easily checked, and for (iii) an intersection-free matching is given by $\{v_2 \leftarrow h_1 \rightarrow v_3, v_6 \leftarrow h_2 \rightarrow v_4\}$. To verify (iv), note that $\pa(\{v\} \cup U) \cap \pa(\{v\} \cup W) = \{h_1, h_2\}$, which implies that there cannot exist an intersection-free matching of $\{v\} \cup U$ and $\{v\} \cup W$.
\item[\underline{$h_2$:}] Take $v=v_2$, $S=\{h_1\}$, $U=\{v_3\}$, $W=\{v_6\}$. The matching $\{v_3 \leftarrow h_2 \rightarrow v_6\}$ is intersection-free, and $(\pa(\{v\} \cup U) \cap \pa(\{v\} \cup W)) \setminus S = \{h_2\}$ implies that (iv) holds.
\item[\underline{$h_3$:}] Take $v=v_3$, $S=\{h_1,h_2\}$, $U=\{v_4\}$, $W=\{v_5\}$. The matching $\{v_4 \leftarrow h_3 \rightarrow v_5\}$ is intersection-free, and $(\pa(\{v\} \cup U) \cap \pa(\{v\} \cup W)) \setminus S = \{h_3\}$ implies that (iv) holds.
\end{itemize}
Note that the graphs in Figure~\ref{fig:intro-example} and \ref{fig:not-AR-identifiable} are also M-identifiable, which can be seen similarly.
\end{example}

Next, we show that M-identifiability subsumes AR-identifiability.

\begin{corollary} \label{cor:subsumes-AR}
    Let $G=(V \cup \cH, D)$ be a factor analysis graph that satisfies ZUTA. Then:
    \begin{itemize}
        \item[(i)] If $G$ is AR-identifiable, then it is also M-identifiable.
        \item[(ii)] If $G$ is full-ZUTA, then $G$ is AR-identifiable if and only if it is M-identifiable.
    \end{itemize}
\end{corollary}

Even though M-identifiability subsumes AR-identifiability, it can also only establish identifiability of graphs that satisfy ZUTA.

\begin{corollary} \label{cor:zuta-satisfied}
Let $G=(V \cup \cH, D)$ be a factor analysis graph that is M-identifiable. Then the factor analysis graph $G$ satisfies ZUTA.
\end{corollary}

\subsection{Extension of the Matching-Criterion}

By Corollary~\ref{cor:subsumes-AR}, M-identifiability subsumes AR-identifiability. However, it does not subsume BB-identifiability. For example, the full-ZUTA graph on $|V|=8$ observed nodes and $|\cH|=4$ latent nodes is BB- but not M-identifiable. We now provide a second criterion that can certify generic sign-identifiability of a set of latent nodes in a way that generalizes BB-identifiability. It operates by searching locally for full-ZUTA subgraphs $\widetilde{G}=(\widetilde{V}, \widetilde{D})$ that satisfy the condition $|\widetilde{V}|+ |\widetilde{D}| < \binom{|\widetilde{V}|+1}{2}$. Combining both criteria then yields an extension of the matching criterion. We start by defining the necessary concepts. 

\begin{definition}
For a set  $B \subseteq V$ of observed nodes, the set of \textit{joint parents} of pairs in $B$ is given by
\[
    \jpa(B) = \{h \in \pa(u) \cap \pa(v): u,v \in B, u \neq v\}.
\]
Moreover, for another set $S \subseteq V$, we say that an ordering $\prec$ on the set $S$ is a \emph{$B$-first-ordering} if, for two elements $v,w \in S$, it holds that $v \prec w$ whenever $v \in B \cap S$ and $w \in S \setminus B$.
\end{definition}

Said differently, a $B$-first-ordering on a set of nodes $S$ is a block-ordering such that all elements in $B$ come first. 

\begin{example}
Consider the graph in Figure~\ref{fig:M-example}, and let $B=\{v_1,v_2,v_3\}$. The joint parents are given by $\jpa(B)=\{h_1,h_2\}$. Moreover, for $S = \{v_1,v_2,v_4,v_5\}$, an example of a $B$-first ordering is given by $v_2 \prec v_1 \prec v_5 \prec v_4$.
\end{example}

We now define a criterion that generalizes BB-identifiability. For $A \subseteq V \cup \cH$, we write $G[A]=(A,D_A)$ for the \emph{induced} subgraph of $G=(V \cup \cH, D)$.
The edge set $D_A=\{h \rightarrow v \in D: h, v \in A\}$ includes precisely those edges in $D$ that have both endpoints in $A$.

\tikzset{
      every node/.style={circle, inner sep=0.3mm, minimum size=0.5cm, draw, thick, black, fill=white, text=black},
      every path/.style={thick}
}
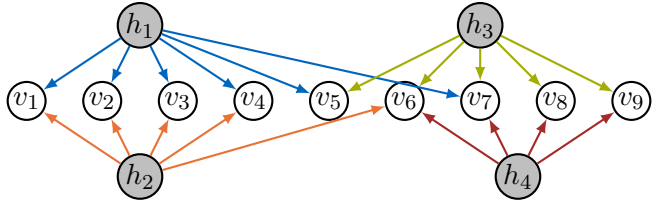
\begin{figure}[t]
\begin{center}
\begin{tikzpicture}[align=center]
    \node[fill=lightgray] (h1) at (-1,1) {$h_1$};
    \node[fill=lightgray] (h2) at (-1,-1) {$h_2$};
    \node[fill=lightgray] (h3) at (3.5,1) {$h_3$};
    \node[fill=lightgray] (h4) at (4,-1) {$h_4$};
    
    \node[] (1) at (-2.5,0) {$v_1$};
    \node[] (2) at (-1.5,0) {$v_2$};
    \node[] (3) at (-0.5,0) {$v_3$};
    \node[] (4) at (0.5,0) {$v_4$};
    \node[] (5) at (1.5,0) {$v_5$};
    \node[] (6) at (2.5,0) {$v_6$};
    \node[] (7) at (3.5,0) {$v_7$};
    \node[] (8) at (4.5,0) {$v_8$};
    \node[] (9) at (5.5,0) {$v_9$};
    
    \draw[MyBlue] [-latex] (h1) edge (1);
    \draw[MyBlue] [-latex] (h1) edge (2);
    \draw[MyBlue] [-latex] (h1) edge (3);
    \draw[MyBlue] [-latex] (h1) edge (4);
    \draw[MyBlue] [-latex] (h1) edge (5);
    \draw[MyBlue] [-latex] (h1) edge (7);
    \draw[MyRed] [-latex] (h2) edge (1);
    \draw[MyRed] [-latex] (h2) edge (2);
    \draw[MyRed] [-latex] (h2) edge (3);
    \draw[MyRed] [-latex] (h2) edge (4);
    \draw[MyRed] [-latex] (h2) edge (6);
    \draw[MyGreen] [-latex] (h3) edge (5);
    \draw[MyGreen] [-latex] (h3) edge (6);
    \draw[MyGreen] [-latex] (h3) edge (7);
    \draw[MyGreen] [-latex] (h3) edge (8);
    \draw[MyGreen] [-latex] (h3) edge (9);
    \draw[MyBrown] [-latex] (h4) edge (6);
    \draw[MyBrown] [-latex] (h4) edge (7);
    \draw[MyBrown] [-latex] (h4) edge (8);
    \draw[MyBrown] [-latex] (h4) edge (9);

\end{tikzpicture}
\caption{Graph that is certified to be generically sign-identifiable by Theorem~\ref{thm:f-identifiability}.}
\label{fig:example-f-id}
\end{center}
\end{figure}

\begin{definition} \label{def:full-factor-criterion}
Let $G=(V \cup \cH, D)$ be a factor analysis graph. We say that the tuple $(B, S) \in 2^V \times 2^{\cH}$ satisfies the \emph{local BB-criterion} if 
\begin{itemize}
    \item[(i)] the induced subgraph $\widetilde{G} = G[B \cup (\jpa(B)\setminus S)]$ is a full-ZUTA graph, 
    \item[(ii)] for all $h \in \jpa(B)\setminus S$, there is a $B$-first-ordering $\prec_h$ on $\ch(h)$ such that for all $v \in \ch(h) \setminus B$ there is $u \in \ch(h)$ with $u \prec_h v$ and $\jpa(\{v,u\}) \setminus S \subseteq \{\ell \in \jpa(B)\setminus S: \ell \preceq_{\text{ZUTA}} h \}$, where $\prec_{\text{ZUTA}}$ is the unique ZUTA-ordering on $\jpa(B)\setminus S$ induced by $\widetilde{G}$,
    \item[(iii)] for the edge set $\widetilde{D}$ of the subgraph $\widetilde{G}$ it holds that $|B| + |\widetilde{D}| < \binom{|B|+1}{2}$.
\end{itemize}
\end{definition}

\begin{theorem}
\label{thm:f-identifiability}
Let $G=(V \cup \cH, D)$ be a factor analysis graph and suppose that the tuple $(B,S)\in 2^V \times 2^{\cH}$ satisfies the local BB-criterion. If all nodes $\ell \in S$ are generically sign-identifiable, then all nodes $h \in \jpa(B)\setminus S$ are generically sign-identifiable. 
\end{theorem}

Similar as for M-identifiability, Theorem~\ref{thm:f-identifiability} allows us to recursively certify generic sign-identifiability of a factor analysis graph by checking whether all nodes $h \in \cH$ are generically sign-identifiable; recall Lemma~\ref{lem:nodes-vs-graph-id}.

\begin{example} \label{ex:F-id}
We can use Theorem~\ref{thm:f-identifiability} to recursively certify generic sign-identifiability of  all latent nodes of the graph displayed in Figure~\ref{fig:example-f-id}.
\begin{itemize}
    \item[\underline{$h_1,h_2$:}] Take $B=\{v_1, \ldots, v_5\}$ and $S=\emptyset$ such that  $\jpa(B)\setminus S = \{h_1,h_2\}$. Observe that  $G[B \cup (\jpa(B)\setminus S)]$ is a full-ZUTA graph such that Condition (iii) is satisfied. Note that the unique ZUTA-ordering  on $\jpa(B)$ is given by $h_1 \prec_{\text{ZUTA}} h_2$, i.e., to verify Condition (ii) we proceed according to this ordering on the latent nodes.  The only child of $h_1$ that is not a member of $B$ is $v_7$. Take any ordering $\prec_{h_1}$ on $\ch(h_1)$ such that $v_7$ is the largest node according $\prec_{h_1}$. Then the ordering $\prec_{h_1}$ is a $B$-first-ordering and $v_1 \prec_{h_1} v_7$. Moreover, observe that $\jpa(\{v_1,v_7\}) = \{h_1\}$. Similarly, we can take any ordering $\prec_{h_2}$ on $\ch(h_2)$ such that $v_6$ is the largest node according $\prec_{h_2}$. Since $\jpa(\{v_1,v_6\}) = \{h_2\}$, we conclude that Condition (ii) is satisfied. \looseness=-1
    \item[\underline{$h_3,h_4$:}] Take $U=\{v_5, \ldots, v_9\}$ and $S=\{h_1,h_2\}$ such that  $\jpa(B)\setminus S = \{h_3,h_4\}$. It is easy to verify that Conditions (i) and (iii) are satisfied. Moreover, we have that $\ch(h_i)\setminus B = \emptyset$ for $i=3,4$, that is, Condition (ii) is trivially satisfied.
\end{itemize}
On the other hand, each observed node in the graph in Figure~\ref{fig:example-f-id} has at least two latent parents. This implies that ZUTA is not satisfied and hence, due to Corollary~\ref{cor:zuta-satisfied}, the graph is not M-identifiable.
\end{example}

Next, we show that the recursive application of Theorem~\ref{thm:f-identifiability} subsumes BB-identifiability, that is, we show equivalence on full-ZUTA graphs. Crucially, Theorem~\ref{thm:f-identifiability} is also able to certify generic sign-identifiability of \emph{sparse} graphs.

\begin{corollary} \label{cor:full-ZUTA-BB-iff-F}
A full-ZUTA graph $G=(V \cup \cH, D)$ is BB-identifiable if and only if generic sign-identifiability of $G$ can be certified by recursively applying Theorem~\ref{thm:f-identifiability}. 
\end{corollary}

\tikzset{
      every node/.style={circle, inner sep=0.3mm, minimum size=0.5cm, draw, thick, black, fill=white, text=black},
      every path/.style={thick}
}
\begin{figure}[t]
\begin{center}
\begin{tikzpicture}[align=center]
    \node[fill=lightgray] (h1) at (-1,1) {$h_1$};
    \node[fill=lightgray] (h2) at (0,-1) {$h_2$};
    \node[fill=lightgray] (h3) at (2,1) {$h_3$};
    \node[fill=lightgray] (h4) at (3,-1) {$h_4$};
    \node[fill=lightgray] (h5) at (5,1) {$h_5$};
    
    \node[] (1) at (-2.5,0) {$v_1$};
    \node[] (2) at (-1.5,0) {$v_2$};
    \node[] (3) at (-0.5,0) {$v_3$};
    \node[] (4) at (0.5,0) {$v_4$};
    \node[] (5) at (1.5,0) {$v_5$};
    \node[] (6) at (2.5,0) {$v_6$};
    \node[] (7) at (3.5,0) {$v_7$};
    \node[] (8) at (4.5,0) {$v_8$};
    \node[] (9) at (5.5,0) {$v_9$};
    \node[] (10) at (6.5,0) {$v_{10}$};
    
    \draw[MyBlue] [-latex] (h1) edge (1);
    \draw[MyBlue] [-latex] (h1) edge (2);
    \draw[MyBlue] [-latex] (h1) edge (3);
    \draw[MyBlue] [-latex] (h1) edge (4);
    \draw[MyBlue] [-latex] (h1) edge (5);
    \draw[MyBlue] [-latex] (h1) edge (6);
    \draw[MyBlue] [-latex] (h1) edge (7);
    \draw[MyBlue] [-latex] (h1) edge (8);
    \draw[MyRed] [-latex] (h2) edge (2);
    \draw[MyRed] [-latex] (h2) edge (3);
    \draw[MyRed] [-latex] (h2) edge (4);
    \draw[MyRed] [-latex] (h2) edge (5);
    \draw[MyRed] [-latex] (h2) edge (6);
    \draw[MyRed] [-latex] (h2) edge (7);
    \draw[MyRed] [-latex] (h2) edge (8);
    \draw[MyGreen] [-latex] (h3) edge (3);
    \draw[MyGreen] [-latex] (h3) edge (4);
    \draw[MyGreen] [-latex] (h3) edge (5);
    \draw[MyGreen] [-latex] (h3) edge (6);
    \draw[MyGreen] [-latex] (h3) edge (7);
    \draw[MyGreen] [-latex] (h3) edge (8);
    \draw[MyBrown] [-latex] (h4) edge (4);
    \draw[MyBrown] [-latex] (h4) edge (5);
    \draw[MyBrown] [-latex] (h4) edge (6);
    \draw[MyBrown] [-latex] (h4) edge (7);
    \draw[MyBrown] [-latex] (h4) edge (8);
    \draw[] [-latex] (h5) edge (7);
    \draw[] [-latex] (h5) edge (9);
    \draw[] [-latex] (h5) edge (10);

\end{tikzpicture}
\caption{Extended M-identifiable sparse factor analysis graph.}
\label{fig:example-c-id}
\end{center}
\end{figure}
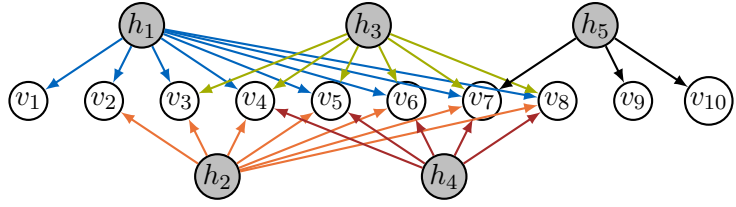

We obtain our final criterion by combining Theorems~\ref{thm:identifiability} and~\ref{thm:f-identifiability} iteratively in a recursive algorithm. We call a factor analysis graph \textit{extended M-identifiable} if generic sign-identifiability can be certified recursively by Theorems~\ref{thm:identifiability} and~\ref{thm:f-identifiability} for all nodes $h \in \cH$. We have already seen in Example~\ref{ex:F-id} that extended M-identifiability may also certify generic sign-identifiability of graphs not satisfying ZUTA. We now provide a further example, where we consider a graph that is extended M-identifiable but applying only one of Theorem~\ref{thm:identifiability} or Theorem~\ref{thm:f-identifiability} does not certify generic sign-identifiability. 

\begin{example}
The factor analysis graph in Figure~\ref{fig:example-c-id} is extended M-identifiable. To see this, we recursively check all latent nodes $\cH=\{h_1, \ldots, h_5\}$.
\begin{itemize}
    \item[\underline{$h_5$:}] The tuple $(v,W,U,S)$ with $v=v_{10}$, $S=\emptyset$, $U=\{v_7\}$ and $W=\{v_9\}$ satisfies the matching criterion with respect to $h_5$. Conditions (i) and (ii) are easily checked, and for Condition (iii) an intersection-free matching is given by $v_7 \leftarrow h_5 \rightarrow v_9$. To verify Condition (iv), note that $\pa(\{v\} \cup U) \cap \pa(\{v\} \cup W)=\{h_5\}$, which implies that there cannot exist an intersection-free matching of $\{v\} \cup U)$ and $\{v\} \cup W$.
    \item[\underline{$\cH \setminus h_5$:}] The tuple $(B,S)$ with $B=\{v_1, \ldots, v_8\}$ and $S=\{h_5\}$ satisfies the local BB-criterion and it holds that $\jpa(B)\setminus S = \cH \setminus h_5$. The induced subgraph $G[B \cup \jpa(B) \setminus S)]$ is a full-ZUTA graph on $8$ observed nodes and $4$ latent nodes for which Condition (iii) holds. Since $\ch(h_i)\setminus B = \emptyset$ for $i=1, \ldots, 4$, Condition (ii) is trivially satisfied.
\end{itemize}
\end{example}

Finally, we emphasize that extended M-identifiability is only sufficient for generic sign-identifiability. Both graphs in Figure~\ref{fig:sign-id-not-c-id} can be shown to be generically sign-identifiable via techniques from computational algebra but are not extended M-identifiable.

\tikzset{
      every node/.style={circle, inner sep=0.3mm, minimum size=0.5cm, draw, thick, black, fill=white, text=black},
      every path/.style={thick}
}
\begin{figure}[t]
\centering
{
(a)
\begin{tikzpicture}[align=center]
    \node[fill=lightgray] (h1) at (-0.5,1) {$h_1$};
    \node[fill=lightgray] (h2) at (0.5,-1) {$h_2$};
    \node[fill=lightgray] (h3) at (1,1) {$h_3$};
    
    \node[] (1) at (-2.5,0) {$v_1$};
    \node[] (2) at (-1.5,0) {$v_2$};
    \node[] (3) at (-0.5,0) {$v_3$};
    \node[] (4) at (0.5,0) {$v_4$};
    \node[] (5) at (1.5,0) {$v_5$};
    \node[] (6) at (2.5,0) {$v_6$};
    
    \draw[MyBlue] [-latex] (h1) edge (1);
    \draw[MyBlue] [-latex] (h1) edge (4);
    \draw[MyBlue] [-latex] (h1) edge (5);
    \draw[MyBlue] [-latex] (h1) edge (6);
    \draw[MyRed] [-latex] (h2) edge (2);
    \draw[MyRed] [-latex] (h2) edge (4);
    \draw[MyRed] [-latex] (h2) edge (5);
    \draw[MyRed] [-latex] (h2) edge (6);
    \draw[MyGreen] [-latex] (h3) edge (3);
    \draw[MyGreen] [-latex] (h3) edge (4);
    \draw[MyGreen] [-latex] (h3) edge (5);
    \draw[MyGreen] [-latex] (h3) edge (6);
\end{tikzpicture}
\hspace{2cm}
(b)
\begin{tikzpicture}[align=center]
    \node[fill=lightgray] (h1) at (-0.5,1) {$h_1$};
    \node[fill=lightgray] (h2) at (0.5,-1) {$h_2$};
    \node[fill=lightgray] (h3) at (1,1) {$h_3$};
    
    \node[] (1) at (-2.5,0) {$v_1$};
    \node[] (2) at (-1.5,0) {$v_2$};
    \node[] (3) at (-0.5,0) {$v_3$};
    \node[] (4) at (0.5,0) {$v_4$};
    \node[] (5) at (1.5,0) {$v_5$};
    \node[] (6) at (2.5,0) {$v_6$};
    
    \draw[MyBlue] [-latex] (h1) edge (1);
    \draw[MyBlue] [-latex] (h1) edge (2);
    \draw[MyBlue] [-latex] (h1) edge (6);
    \draw[MyRed] [-latex] (h2) edge (2);
    \draw[MyRed] [-latex] (h2) edge (3);
    \draw[MyRed] [-latex] (h2) edge (4);
    \draw[MyRed] [-latex] (h2) edge (5);
    \draw[MyRed] [-latex] (h2) edge (6);
    \draw[MyGreen] [-latex] (h3) edge (3);
    \draw[MyGreen] [-latex] (h3) edge (4);
    \draw[MyGreen] [-latex] (h3) edge (5);
    \draw[MyGreen] [-latex] (h3) edge (6);
\end{tikzpicture}
}
\caption{Two generically sign-identifiable graphs that are not extended M-identifiable.}
\label{fig:sign-id-not-c-id}
\end{figure}
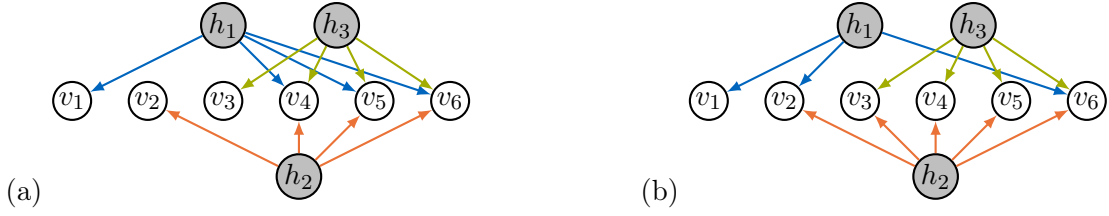

\section{Computation} \label{sec:computation}
M-identifiability and extended M-identifiability are decidable in polynomial time under certain bounds on the sizes of the subsets involved. In Appendix~\ref{sec:algos}, we detail efficient algorithms  that are sound and complete.

\begin{theorem}\label{thm:decidable-M}
M-identifiability of a factor analysis graph $G=(V \cup \mathcal{H}, D)$ is decidable in time $\mathcal{O}(|\cH|^{2} |V|^{k+1} (|V|+|\cH|)^3)$ if we only allow sets $W$ with $|W| \leq k$ in the matching criterion.
\end{theorem}
\begin{proof}
    See Theorem~\ref{thm:algorithm-one-node} and Algorithm~\ref{alg:check-M-id} in Appendix~\ref{sec:algos}.
\end{proof}

\begin{theorem} \label{thm:decidable-extended-M}
Extended M-identifiability of a factor analysis graph $G=(V \cup \mathcal{H}, D)$ is decidable in time $\mathcal{O}(|\cH|^2|V|^{\max\{k,l\}+1} (|V|+|\cH|)^3)$ if we only allow sets $W$ with $|W| \leq k$ in the matching criterion and only sets $B$ with $|B| \leq \ell$ in the local BB-criterion.
\end{theorem}
\begin{proof}
    See Theorem~\ref{thm:algo-F-id} and Algorithm~\ref{alg:check-extended-M-id} in Appendix~\ref{sec:algos}.
\end{proof}

If we allow the cardinality of the sets to be unbounded, then the algorithms we propose in the Appendix search over an exponentially large space and, thus, may take exponential time in the number of nodes. We conjecture that one cannot do significantly better.

\begin{conjecture} \label{conj:M-NP-hard}
    Deciding M-identifiability and extended M-identifiability both is NP-complete.
\end{conjecture}

\begin{remark}
In practice, if there are no restrictions in terms of computational time, then we can allow sets $W$ and $B$ of arbitrary size. In this case,  Algorithm~\ref{alg:check-extended-M-id} in Appendix~\ref{sec:algos}  is sound and complete for deciding extended M-identifiability of a latent-factor graph. That is, it returns ``yes'' if and only if the input graph is extended M-identifiable, see Theorem~\ref{thm:algo-F-id}. However, the run-time of the unconstrained algorithm is exponential in the number of nodes of the graph. Note that allowing sets of $W$ and $B$ of unconstrained size is equivalent to choosing $k=|\cH|$ and $\ell=|V|$ since $|W|\leq|\cH|$ and $|B|\leq|V|$ according to the definitions of the matching criterion and the local BB-criterion. Choosing smaller maximal sizes $k$ and $\ell$ can be useful in practice when attempting to verify generic sign-identifiability of large graphs, where the unconstrained version of Algorithm~\ref{alg:check-extended-M-id} does not terminate in a reasonable amount of time. With $k<|\cH|$ and $\ell < |V|$, Algorithm~\ref{alg:check-extended-M-id} is sound but not complete. That is, if the Algorithm returns ``yes'' with $k<|\cH|$ and $\ell < |V|$, then extended M-identifiability holds, which in turn implies generic sign-identifiability. In this case, during the recursive computations, every tuple certified to satisfy the matching criterion fulfills $|W|\leq k$, and every tuple certified to satisfy the local BB-criterion fulfills $|B|\leq \ell$. However, if the Algorithm returns ``no'' with $k<|\cH|$ and $\ell < |V|$, then we remain inconclusive whether the input graph is extended M-identifiable. 
\end{remark}

\begin{remark}
\citet{hosszejni2026cover} provide an efficient method to check AR-identifiability in polynomial time by computing maximal matchings in a bipartite graph. As explained in Remark~\ref{rem:comparison-matchings}, their approach is infeasible for checking our matching criterion, as it is a local version of AR-identifiability. The reason why the matching criterion is significantly more computationally intensive is as follows. Recall from Corollary~\ref{cor:ar-equivalence} that a factor analysis graph that satisfies ZUTA is AR-identifiable if and only if, for all $v \in V$, there exist two disjoint sets $W,U \subseteq V \setminus \{v\}$ with $|W|=|U|=|\cH|$ and an intersection-free matching between $W$ and $U$. The crucial difference to Condition (iii) in the matching criterion is that it is already known a priori that the intersection-free matching will involve all nodes $h \in \cH$. Checking Condition (iii) in the matching criterion can be seen as checking AR-identifiability locally for every possible subset of latent variables $H \subseteq \cH$. 
Moreover, the matching criterion needs the additional Condition (iv)  to avoid an intersection free matching between  $\{v\} \cup W$ and $\{v\} \cup U$. This is not needed for AR-identifiability because the existence of such a matching is impossible if  $|W|=|U|=|\cH|$.
\end{remark}

\begin{table*}[t]\centering 
\begin{tabular}{@{}c|rrrrr@{}}\toprule
nr. edges & ZUTA & gen. sign-id & AR-id & M-id & Ext.~M-id\\ 
\midrule
1 & 1 & 0 & 0 & 0 & 0 \\
2 & 2 & 0 & 0 & 0 & 0 \\
3 & 4 & 1 & 1 & 1 & 1 \\
4 & 7 & 1 & 1 & 1 & 1 \\
5 & 14 & 1 & 1 & 1 & 1 \\
6 & 25 & 3 & 3 & 3 & 3 \\
7 & 41 & 4 & 4 & 4 & 4 \\
8 & 56 & 6 & 6 & 6 & 6 \\
9 & 73 & 8 & 7 & 8 & 8 \\
10 & 77 & 11 & 9 & 11 & 11 \\
11 & 79 & 23 & 16 & 23 & 23 \\
12 & 67 & 31 & 23 & 29 & 29 \\
13 & 54 & 33 & 29 & 33 & 33 \\
14 & 31 & 23 & 21 & 23 & 23 \\
15 & 18 & 16 & 16 & 16 & 16 \\
16 & 8 & 8 & 8 & 8 & 8 \\
17 & 4 & 4 & 4 & 4 & 4 \\
18 & 1 & 1 & 1 & 1 & 1 \\
\midrule
Total & 562 & 174 & 150 & 172 & 172 \\
\bottomrule
\end{tabular}
\caption{Counts of unlabeled sparse factor graphs satisfying ZUTA with at most $|V|=7$ observed nodes and $|\cH|=3$ latent nodes.}
\label{table:3-7}
\end{table*}

\section{Numerical Experiments} \label{sec:experiments}

In this section, we first conduct simulations that demonstrate the performance of our criteria. Then we exemplify how our identifiability criteria are also useful in exploratory factor analysis.

\subsection{Simulations}

In our simulations, we compare different criteria for generic sign-identifiability. We treat graphs as unlabeled, that is, we count isomorphism classes of graphs. Two factor analysis graphs $G=(V \cup \cH, D)$ and $G'=(V \cup \cH, D')$ on the same set of nodes are isomorphic if there is a permutation $\pi_V$ on the observed nodes $V$ and a permutation $\pi_{\cH}$ on the latent nodes $\cH$ such that, for $h \in \cH$ and $v \in V$, the edge $h \rightarrow v \in D$ if and only if $\pi_{\cH}(h) \rightarrow \pi_{V}(v) \in D'$.

In our first experimental setup, we consider factor analysis graphs with a small number of observed and latent nodes where generic sign-identifiability can be fully solved by methods from computational algebraic geometry, as we discuss in Appendix~\ref{sec:computational-algebra}. Table~\ref{table:3-7} lists counts of all factor analysis graphs with up to $3$ latent nodes and $7$ observed nodes that satisfy ZUTA. We count how many of the graphs are AR-identifiable, M-identifiable, and extended M-identifiable. For deciding AR-identifiability we use the algorithm and code provided by \citet{hosszejni2026cover}. Note that our criteria are very effective since we only fail to certify generic sign-identifiability of two graphs. Those are exactly the graphs displayed in Figure~\ref{fig:sign-id-not-c-id}. Table~\ref{table:3-7} also illustrates that M-identifiability subsumes AR-identifiability as we have shown in Corollary~\ref{cor:subsumes-AR}. On the other hand, M-identifiability and extended M-identifiability coincide on the considered set of small graphs. This is as expected since the smallest graph, where BB-identifiability holds but M-identifiability does not, has $8$ observed nodes and $4$ latent nodes.

\begin{table}[t!]\centering 
\begin{tabular}{@{}c|rrrrr@{}}\toprule
nr. edges & total & ZUTA & AR-id & M-id & Ext.~M-id\\ 
\midrule
1 & 1 & 1 & 0 & 0 & 0  \\ 
2 & 3 & 2 & 0 & 0 & 0  \\ 
3 & 6 & 4 & 1 & 1 & 1  \\ 
4 & 16 & 8 & 1 & 1 & 1  \\ 
5 & 27 & 16 & 1 & 1 & 1  \\ 
6 & 62 & 34 & 3 & 3 & 3  \\ 
7 & 111 & 71 & 4 & 4 & 4  \\ 
8 & 225 & 146 & 9 & 9 & 9  \\ 
9 & 395 & 287 & 14 & 15 & 15  \\ 
10 & 716 & 528 & 21 & 23 & 23  \\ 
11 & 1165 & 922 & 43 & 50 & 50  \\ 
12 & 1880 & 1504 & 81 & 93 & 93  \\ 
13 & 2726 & 2273 & 134 & 167 & 167  \\ 
14 & 3829 & 3192 & 221 & 366 & 368  \\ 
15 & 4890 & 4147 & 404 & 768 & 768  \\  
16 & 5963 & 4972 & 759 & 1430 & 1435  \\ 
17 & 6599 & 5490 & 1299 & 2187 & 2204  \\ 
18 & 6937 & 5519 & 1927 & 2861 & 2913 \\
19 & 6599 & 5047 & 2385 & 3164 & 3273  \\ 
20 & 5963 & 4191 & 2509 & 3037 & 3179  \\ 
21 & 4890 & 3157 & 2231 & 2520 & 2656  \\ 
22 & 3829 & 2139 & 1705 & 1833 & 1913  \\ 
23 & 2726 & 1310 & 1128 & 1177 & 1215  \\ 
24 & 1880 & 710 & 651 & 665 & 683  \\ 
25 & 1165 & 343 & 328 & 331& 344  \\ 
26 & 716 & 153 & 151 & 151& 155  \\ 
27 & 395 & 64 & 64 & 64& 65  \\  
28 & 225 & 22 & 22 & 22 & 22  \\ 
29 & 111 & 7 & 7 & 7 & 7  \\ 
30 & 62 & 1 & 1 & 1 & 1  \\
$>$30 & 54 & 0 & 0 & 0 & 0 \\
\midrule
Total & 64166 & 46260 & 16104 & 20951 & 21568 \\
\bottomrule
\end{tabular}
\caption{Counts of unlabeled sparse factor graphs with at most $|V|=9$ observed nodes and $|\cH|=4$ latent nodes.}
\label{table:4-9}
\end{table}

Our second experimental setup considers all factor analysis graphs with up to $4$ latent nodes and $9$ observed nodes, and also includes graphs that do not satisfy ZUTA. 
Table~\ref{table:4-9} shows that the gap between AR-identifiability and M-identifiability increases. Moreover, extended M-identifiability indeed becomes effective since there are 617 graphs that are extended M-identifiable but not M-identifiable. Recall also that graphs not satisfying ZUTA might be extended M-identifiable but they are never M-identifiable nor AR-identifiable. For computational reasons, we are not able to fully solve generic sign-identifiability by methods from computational algebraic geometry for all graphs considered in Table~\ref{table:4-9}.

Next, to demonstrate that checking extended M-identifiability is feasible on larger graphs with more observed and latent nodes, we randomly generate graphs on $25$ observed nodes and $10$ latent nodes. We draw the graphs from an Erdös-Renyi model with edge probabilities $0.2$, $0.25$, $0.3$, $0.35$, $0.4$, $0.45$, where we fix the upper triangle of the adjacency matrix to zero to increase the probability of satisfying ZUTA. Moreover, we only consider graphs with at most $10$ children per latent node such that the maximal cardinality of a set $B$ satisfying the local BB-criterion is at most $10$. For each edge probability, we sample $5000$ graphs and check whether they are extended M-identifiable. When searching for sets that satisfy the matching criterion we bound the size of set $W$ by $k=4$. The number of graphs that were extended M-identifiable is reported in Table~\ref{table:large-graphs}. Recall that only graphs with at least three children per latent node can be extended M-identifiable.  For low edge probabilities the likelihood is high that this is not satisfied. As expected the fraction of extended M-identifiable graphs increases with increasing edge probabilities. However, at a certain density level we would expect that fewer graphs are extended M-identifiable, which we can already see for the edge probability $p=0.45$. \looseness=-1

\begin{table}[t]\centering 
\begin{tabular}{@{}c|rrrrrr@{}}
\toprule
Edge probability & 0.2 & 0.25 & 0.3 & 0.35 & 0.4 & 0.45\\ 
\midrule
Ext.~M-identifiable & 510 & 1784 & 3234 & 4184 & 4536 & 4498 \\
\midrule
Percentage & 10.2  & 35.7  & 64.7  & 83.7  & 90.7  & 90.0  \\
\bottomrule
\end{tabular}
\caption{Status of extended M-identifiability for 5000 randomly generated sparse factor graphs with different edge probabilities with $|W|\leq k$ for $k=4$.}
\label{table:large-graphs}
\end{table}

\subsection{Application to Exploratory Factor Analysis}

In this section, we discuss how our identifiability criteria are also useful in exploratory factor analysis. It is a desirable property in exploratory factor analysis to discover a sparse structure that yields interpretable factor loadings. If we apply threshold-based sparse estimation methods, for example, our identifiability criteria can provide guidance in choosing the threshold or tuning parameter such that identifiability is ensured.

To exemplify this in a small case study, we consider the 2018 Populism and Political Parties Expert Survey (POPPA) that measures positions and attitudes of 250 parties on key attributes related to populism, political style, party ideology, and party organization \citep{meijers2020data}. The data set is obtained from an expert survey in 28 European countries and contains $|V|=15$ measured variables. After discarding data points with missing values, $220$ samples remain. In their analyses of the data, \citet{meijers2021measuring}  also conduct an exploratory factor analysis. While they retain two latent factors in the main manuscript, they also consider five factors in the supplement. We replicate their study with five factors by first estimating the loading matrix via maximum likelihood and then applying varimax rotation \citep{kaiser1958varimax} using  the \texttt{factanal} function in the \texttt{base} library of \texttt{R} \citep{R}. We then set all loadings smaller than a predefined threshold to zero to obtain a sparse loading matrix $\Lambda$. The associated factor analysis graph includes edge $h \rightarrow v \in D$ if and only if $\lambda_{hv} \neq 0$.  Note that we do not analyze the two-factor model because generic sign-identifiability is readily ensured as long as at least one measurement does not load on both factors, making the five-factor model a more illustrative setting.

In their analysis, \citet{meijers2021measuring} use a threshold of $0.5$, which yields a very sparse graph with some factors having less than three children, and thus generic sign-identifiability does not hold; recall Remark~\ref{rem:dimension}. In contrast, we check whether extended M-identifiability holds for graphs obtained from different thresholds.  We obtain that extended M-identifiability only holds for thresholds in the interval $[0.10, 0.14]$. Note that ZUTA is not satisfied for the graphs given by all thresholds in this interval, while it is satisfied for all graphs obtained from thresholds $\geq 0.15$. In Table~\ref{table:real-data} we plot the factor loading matrix that we obtain for the threshold $0.10$.  \citet{meijers2021measuring} argue that the first latent factor represents populism since the first five measurements were designed to measure populism and load strongly on it. For more information on the measured variables we refer to \citet{meijers2021measuring}.

\begin{table*}[t]\centering 
\begin{tabular}{@{}c|rrrrr@{}}\toprule
 & Factor 1 & Factor 2 & Factor 3 & Factor 4 & Factor 5\\ 
\midrule
1      &  0.799 &  0.230 &  0.240 &  0.141 &  0.138 \\  
2      &  0.647 &  0.468 &  0.264 &  0.375 &        \\  
3      &  0.735 &  0.356 &  0.286 &  0.495 &        \\  
4      &  0.921 &        &        &  0.147 &        \\  
5      &  0.960 &        &        &        &        \\  
6      & -0.424 &  0.541 &        &        &  0.331 \\  
7      & -0.242 & -0.894 & -0.192 &        & -0.107 \\  
8      & -0.737 & -0.380 & -0.146 &  0.113 &  0.149 \\  
9      &  0.357 &  0.866 &  0.185 &  0.114 &  0.129 \\  
10     &  0.238 &  0.887 &  0.202 &        &        \\  
11     & -0.172 & -0.861 & -0.212 &        &        \\  
12     & -0.843 & -0.265 & -0.263 &        &        \\  
13     &  0.798 &  0.302 &  0.145 &  0.120 &  0.316 \\  
14     & -0.314 & -0.486 & -0.809 &        &        \\  
15     &  0.289 &  0.422 &  0.506 &        &  0.413 \\  
\bottomrule
\end{tabular}
\caption{Factor loading matrix obtained via maximum likelihood estimation and varimax rotation from the POPPA data set. Loadings with absolute value $< 0.1$ are not shown.}
\label{table:real-data}
\end{table*}

Note that our identifiability conditions can also be incorporated in Bayesian sparse factor analysis, where sparse structures are discovered by employing ``spike-and-slab'' priors on the factor loading matrix  \citep{fruehwirth2025sparse, hosszejni2026cover, conti2014bayesian}. In this case, generic sign-identifiability may be imposed as a domain restriction on the parameter space of the prior distribution. In practice, since the posterior distribution is typically obtained via MCMC sampling, this amounts to post-processing posterior draws under the unrestricted prior by discarding draws whose sparsity pattern does not permit generic sign-identifiability.

\section{Discussion}
We introduced a formal graphical framework for studying identifiability in confirmatory factor analysis when the factor loading matrix is sparse. Our main results provide graphical criteria that constitute sufficient conditions for generic sign-identifiability.  It is worth mentioning that even if a model is not extended M-identifiable, it may still be possible to prove generic sign-identifiability of certain columns of the factor loading matrix. This is the case if the recursive algorithm stops early declaring only some but not all latent nodes $h \in \cH$ to be generically sign-identifiable. 

Generic sign-identifiability is useful if an interpretation of the latent factors and their effects on the observed variables is desired. Moreover, if a model is identifiable in this sense, then its dimension equals the expected dimension obtained from counting parameters. This is crucial information for goodness-of-fit tests and model selection procedures.

Our work opens up some natural questions for further studies. 
For instance, in generalization of the setup we studied in this paper, one may also consider factor analysis models where the factors itself may be correlated. Then the observed covariance matrix takes the form
$
    \Sigma = \Lambda \Phi \Lambda^{\top} +\Omega
$
for a positive definite matrix $\Phi$ that may also be sparse. For example, consider the graph in Figure~\ref{fig:dependent-factors}, where the bidirected edge represents the nonzero correlation between the latent factors. The parameter matrices are then given by \looseness=-1
\[
    \Lambda = \begin{pmatrix}
        \lambda_{11} & 0 \\
        \lambda_{21} & 0 \\
        0 & \lambda_{32} \\
        0 & \lambda_{42}
    \end{pmatrix}, \quad
    \Phi = \begin{pmatrix}
        1 & \phi_{12} \\
        \phi_{12}&  1
    \end{pmatrix}, \quad \text{and} \quad
    \Omega = \begin{pmatrix}
        \omega_{11} & 0 & 0 & 0 \\
        0 & \omega_{22} & 0 & 0 \\
        0 & 0 & \omega_{33} & 0 \\
        0 & 0 & 0 & \omega_{44} 
    \end{pmatrix}.
\]
The observed covariance matrix takes the form
\[
\Sigma = (\sigma_{uv}) = \begin{pmatrix}
    \omega_{11} + \lambda_{11}^2 & \lambda_{11} \lambda_{21} & \lambda_{11} \phi_{12} \lambda_{32} & \lambda_{11} \phi_{12} \lambda_{42} \\
    \lambda_{11} \lambda_{21} & \omega_{22} + \lambda_{21}^2 & \lambda_{21} \phi_{12} \lambda_{32} & \lambda_{21} \phi_{12} \lambda_{42} \\
    \lambda_{11} \phi_{12} \lambda_{32} & \lambda_{21} \phi_{12} \lambda_{32} & \omega_{33} + \lambda_{32}^2 & \lambda_{32} \lambda_{42} \\
    \lambda_{11} \phi_{12} \lambda_{42} & \lambda_{21} \phi_{12} \lambda_{42} & \lambda_{32} \lambda_{42} & \omega_{44} + \lambda_{42}^2
\end{pmatrix}.
\]

\tikzset{
      every node/.style={circle, inner sep=0.3mm, minimum size=0.5cm, draw, thick, black, fill=white, text=black},
      every path/.style={thick}
}
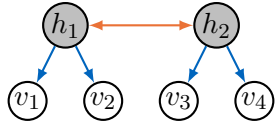
\begin{figure}[t]
\begin{center}

\begin{tikzpicture}[align=center]
    \node[fill=lightgray] (h1) at (-2,1) {$h_1$};
    \node[fill=lightgray] (h2) at (0,1) {$h_2$};
    
    \node[] (1) at (-2.5,0) {$v_1$};
    \node[] (2) at (-1.5,0) {$v_2$};
    \node[] (3) at (-0.5,0) {$v_3$};
    \node[] (4) at (0.5,0) {$v_4$};
    
    \draw[MyBlue] [-latex] (h1) edge (1);
    \draw[MyBlue] [-latex] (h1) edge (2);
    \draw[MyBlue] [-latex] (h2) edge (3);
    \draw[MyBlue] [-latex] (h2) edge (4);
    \draw[MyRed] [latex-latex] (h1) edge (h2);
\end{tikzpicture}
\end{center}
\caption{Graph encoding a nonzero correlation between the latent factors.}
\label{fig:dependent-factors}
\end{figure}

\noindent
It is already noted in \citet[p.~245]{bollen1989structural} that the parameters of this model are identifiable up to sign. Generically, the parameter  $\phi_{12}$ is recovered up to sign via the formula
\[
    \sqrt{\frac{\sigma_{23} \sigma_{14}}{\sigma_{12} \sigma_{34}}} = \sqrt{\frac{\lambda_{21} \phi_{12} \lambda_{32} \lambda_{11} \phi_{12} \lambda_{42}}{\lambda_{11} \lambda_{21} \lambda_{32} \lambda_{42}}} = |\phi_{12}|.
\]
Given $|\phi_{12}|$, we obtain generic identifiability of $|\lambda_{11}|$ via 
\[
\sqrt{\frac{\sigma_{13} \sigma_{14}}{\sigma_{34}|\phi_{12}|^2}} = \sqrt{\frac{\lambda_{11} \phi_{12} \lambda_{32} \lambda_{11} \phi_{12} \lambda_{42}}{\lambda_{32} \lambda_{42} \phi_{12}^2}} = |\lambda_{11}|,
\]
and all other parameters can be recovered similarly. Notably, since each latent node only has two children, generic sign-identifiability is impossible if the correlation $\phi_{12}$ is zero, as we have also seen in Remark~\ref{rem:dimension}. We believe that our work provides tools for future work on deriving conditions for generic sign-identifiability in models with dependent factors.

Another interesting future research direction is to study structure identifiability, a topic of interest in exploratory factor analysis. Structure identifiability refers to studying whether knowledge of the covariance matrix allows one to uniquely recover an unknown underlying graph. For model selection methods that return sparse factor loading matrices, one might then derive guarantees for recovering a most parsimonious true graph.

\section*{Acknowledgments}
The project has received funding from the European Research Council (ERC) under the European Union’s Horizon 2020 research and innovation programme (grant agreement No 883818) and from the German Federal Ministry of Education and Research and the Bavarian State Ministry for Science and the Arts.

\bibliographystyle{apalike}
\bibliography{bibliography}

\newpage
\begin{appendix}
\begin{center}
{ \LARGE\bf \noindent Appendix} 
\end{center}

 \section{Additional Lemmas for Full Factor Models} \label{sec:full-factor}

We relate the existing literature on identifiability of full factor analysis models  to our setup that also allows for sparse factor models. To formally define generic identifiability of the diagonal matrix $\Omega$, we use the same terminology as \citet{bekker1997generic}. Let $\pi_{\diag}$ be the projection of the parameters space $\Theta_G$ to the parameters corresponding to the diagonal matrix, that is,
\begin{align} \label{eq:pi-diag}
    \pi_{\diag}: \Theta_G &\longrightarrow \mathbb{R}^{|V|}_{>0} \\
    (\Omega, \Lambda) &\longmapsto \Omega. \nonumber
\end{align}

\begin{definition} \label{def:global-identifiability}
A factor analysis graph $G=(V \cup \cH, D)$ is said to be \emph{generically globally identifiable} if $\pi_{\diag}(\mathcal{F}_G(\Omega, \Lambda)) = \{\Omega\}$ for almost all $(\Omega, \Lambda) \in \Theta_G$.
\end{definition}

We have the following lemma.

\begin{lemma} \label{lem:global-id-follows-sign-id}
Let $G=(V \cup \mathcal{H}, D)$ be a factor analysis graph such that ZUTA is satisfied. Then $G$ is generically globally identifiable if and only if $G$ is  generically sign-identifiable.
\end{lemma} 
\begin{proof}
If $G$ is generically sign-identifiable, then by definition it is generically globally identifiable. For the other direction, let $G$ be generically globally identifiable.  Fix a generically chosen parameter tuple $(\Omega, \Lambda) \in \Theta_G$ and consider a tuple $(\widetilde{\Omega}, \widetilde{\Lambda}) \in \mathcal{F}_G(\Omega, \Lambda)$ in the fiber of $(\Omega, \Lambda)$. It holds that $\Lambda \Lambda^{\top} + \Omega = \widetilde{\Lambda} \widetilde{\Lambda}^{\top} + \widetilde{\Omega}$. Since $G$ is generically globally identifiable, it follows that $\Omega = \widetilde{\Omega}$, and therefore  $\Lambda \Lambda^{\top}= \widetilde{\Lambda} \widetilde{\Lambda}^{\top}$. We assume w.l.o.g.~that each latent node has at least one child since otherwise the node is trivially generically-sign identifiable. Since ZUTA is satisfied it holds that $\Lambda$ generically has full column rank. By multiplying from the right with the Moore-Penrose pseudoinverse of $\widetilde{\Lambda}^{\top}$, it must hold that $\widetilde{\Lambda} = \Lambda Q$ for some invertible matrix $Q \in \mathbb{R}^{|\cH| \times |\cH|}$. Moreover, this matrix $Q$ has to be orthogonal, since otherwise  the equality $\Lambda \Lambda^{\top} = \widetilde{\Lambda} \widetilde{\Lambda}^{\top} = \Lambda Q Q^{\top} \Lambda^{\top}$ does not hold; to see this multiply with the pseudoinverses of $\Lambda$ and $\Lambda^{\top}$ again.

For finishing the proof, we claim that $\Lambda Q \in \mathbb{R}^D$ for an orthogonal matrix $Q$ if and only if $Q$ is diagonal with entries in  $\{\pm 1\}$. Assume w.l.o.g.~that the latent nodes are ordered by the ZUTA-ordering. Then, there also exists an ordering of the observed nodes such that $\Lambda$ contains a lower triangular submatrix $\Lambda_{U,\cH}$ with no zeros on the diagonal for a subset $U \subseteq V$ with $|U|=|\cH|$. We can choose the subset $U$ such that, for generically chosen $\Lambda$, the lower triangular matrix $\Lambda_{U,\cH}$ is nonsingular. Observe that $\Lambda Q \in \mathbb{R}^D$ implies that the matrix $T := \Lambda_{U,\cH} Q$ also has to be lower triangular. Hence, $Q = \Lambda_{U,\cH}^{-1}T$ is also lower triangular. But then $Q^{-1}=Q^{\top}$  has to be lower triangular as well, which implies that $Q$ is diagonal. Finally, note that $I=Q^{\top}Q$ which implies that the squared diagonal entries of $Q$ are all equal to one.
\end{proof}

Existing literature mainly focused on full factor models, which correspond to the following graphs.

\begin{definition} \label{def:full-factor-graph}
    A \emph{full} factor analysis graph is a factor analysis graph $G=(V \cup \cH, D)$ with $D=\cH \times V$.
\end{definition}

Now, we show that every full factor analysis graph is in a one-to-one relation to the full-ZUTA graph on the same set of nodes.

\begin{lemma}\label{lem:full-factor-models}
    Let $G=(V \cup \cH, \cH \times V)$ be a full factor analysis graph and consider the corresponding full-ZUTA graph $G'=(V \cup \cH, D)$ on the same set of nodes. Then $F(G)=F(G')$ and, moreover, $G$ is generically globally identifiable if and only if $G'$ is generically sign-identifiable.
\end{lemma} 
\begin{proof}
    The inclusion $F(G') \subseteq F(G)$ is trivial since the edges of $G'$ are a subset of the edges of $G$. For the other inclusion, consider a point $\Sigma \in F(G)$ in the full factor analysis model. Then $\Sigma = \Lambda \Lambda^{\top} + \Omega$  for some parameters $(\Omega, \Lambda) \in \Theta_G$.  As in the QR-decomposition, there is an orthogonal matrix $Q$ such that $\Lambda Q$ has a zero upper triangle, that is, $\Lambda Q \in \mathbb{R}^{D}$ where $D$ is the edge set of the full-ZUTA graph $G'$. Since $\Lambda Q Q^{\top} \Lambda^{\top} + \Omega = \Lambda \Lambda^{\top} + \Omega = \Sigma$, we conclude that $\Sigma \in F(G')$.

    For the second statement, observe that $G$ is generically globally identifiable if and only if $G'$ is generically globally identifiable. Since $G'$ satisfies ZUTA, it follows by Lemma~\ref{lem:global-id-follows-sign-id} that $G'$ is generically globally identifiable if and only if $G'$ generically sign-identifiable.
\end{proof}

\section{Algorithms} \label{sec:algos}
In this section, we propose efficient algorithms for deciding whether a sparse factor analysis graph is M-identifiable or extended M-identifiable.

\subsection{Deciding M-identifiability}
Let $G=(V \cup \cH, D)$ be a factor analysis graph and fix a node $h \in \cH$. When recursively checking M-identifiability, the next lemma verifies that we can always take the set $S \subseteq \cH \setminus \{h\}$ to be the set of \emph{all} previously solved nodes, that is, all latent nodes that are already known to be generically sign-identifiable.

\begin{lemma} \label{lem:choice-of-S}
Fix a latent node $h \in \cH$ in a factor analysis graph $G=(V \cup \mathcal{H}, D)$.  Let $\widetilde{S} \subseteq \cH \setminus \{h\}$, and suppose that $(v, W, U, S) \in V \times 2^{V} \times 2^{V} \times 2^{ \cH \setminus \{h\}}$ with $S \subseteq \widetilde{S}$ satisfies the matching criterion with respect to $h$. Then there are $\widetilde{W},\widetilde{U} \subseteq V$ such that the tuple $(v, \widetilde{W}, \widetilde{U}, \widetilde{S})$ also satisfies the matching criterion with respect to $h$.
\end{lemma}
\begin{proof}
Define $W=\{w_1, \ldots, w_k\}$ and $U=\{u_1, \ldots, u_k\}$ and let $\Pi=\{\pi_1, \ldots, \pi_k\}$ be an intersection-free matching of $W$ and $U$ such that $\pi_i$ is given by $w_i \leftarrow h_i \rightarrow u_i$. Since $\Pi$ avoids $S$, each latent node $h_i$ that appears in $\Pi$ is not in $S$.  If a latent node $h_i$ is in $\widetilde{S}$, we remove $w_i$ from $W$ and  $u_i$ from $U$. This defines $\widetilde{W}$ and $\widetilde{U}$ as subsets of $W$ and $U$ respectively. Now, we check that the tuple $(v, \widetilde{W}, \widetilde{U}, \widetilde{S})$ satisfies conditions (i) - (iv) of the matching criterion; recall Definition~\ref{def:matching-criterion}.

To check condition (i), recall that $h \not\in \widetilde{S}$ and $S \subseteq \widetilde{S}$. Hence, it directly follows from $\pa(v) \setminus S = \{h\}$ that $\pa(v) \setminus \widetilde{S} = \{h\}$ also holds. Moreover, we have that $v \not\in \widetilde{W} \cup \widetilde{U}$ since $v \not\in W \cup U$. 

To check condition (ii), first note that $\widetilde{W}$ and $\widetilde{U}$ are subsets of the disjoint sets $W$ and $U$. Thus, the sets $\widetilde{W}$ and $\widetilde{U}$ are also disjoint. To see that $\widetilde{W}$ and $\widetilde{U}$ are nonempty, observe that by the definition of the matching criterion, $h$ is equal to one of the latent nodes $h_1, \ldots, h_k$ appearing in $\Pi$, say $h= h_j$ for some $j \in [k]$. Since $h \not\in \widetilde{S}$, we have not removed $w_j$ and $u_j$ from $W$ and $U$. Hence, the sets $\widetilde{W}$ and $\widetilde{U}$ are nonempty. Clearly, they also have equal cardinality.  

To check condition (iii) note that the paths $\pi_i \in \Pi$ that do not visit a latent node $h_i \in \widetilde{S}$ define a intersection-free matching of $\widetilde{W}$ and $\widetilde{U}$.

Finally, it can be seen by contradiction that condition (iv) is also satisfied. Suppose there is an intersection free matching of $\{v\} \cup \widetilde{W}$ and $\{v\} \cup \widetilde{U}$ that avoids $\widetilde{S}$.  Adding the paths $\pi_i: w_i \leftarrow h_i \rightarrow u_i$ with $h_i \in \widetilde{S} \setminus S$ to this matching, gives an intersection-free matching of $\{v\} \cup W$ and $\{v\} \cup U$. This is a contradiction and we conclude that there can not exist an intersection-free matching of $\{v\} \cup \widetilde{W}$ and $\{v\} \cup \widetilde{U}$ that avoids $\widetilde{S}$.
\end{proof}

For a given node $h \in \cH$ and a set $S \subseteq \cH \setminus \{h\}$, we denote by \texttt{M}($G,h,S$) the decision problem whether there exists a tuple $(v, W, U) \in V \times 2^{V} \times 2^{V}$ such that $(v, W, U, S)$ satisfies the matching criterion  with respect to $h$. To decide  \texttt{M}($G,h,S$), we make use of maximum flows in a special flow graph $G_{\flow}=(V_f,D_f)$ from a designated source node $s \in V_f$ to a target node $t \in V_f$. The standard maximum-flow framework is introduced in \citet{cormen2009introduction}, and a maximum-flow in an acyclic directed graph can be computed in polynomial time with complexity $\mathcal{O}(|V_f|^3)$.

We first address the subproblem of verifying whether a given tuple $(v, W, U, S) \in V \times 2^{V} \times 2^{V} \times 2^{ \cH \setminus \{h\}}$ satisfies the matching criterion with respect to $h$. Conditions (i)-(ii) in the definition of the matching criterion are easy to check. If we suppose that they are satisfied, then we are able to check whether Condition (iii) holds, i.e., whether there exists an intersection-free matching of $W$ and $U$ that avoids $S$, by one maximum flow computation on a suitable flow graph $G_{\flow}^{(iii)}(v,W,U,S)=(V_{(iii)},D_{(iii)})$. The nodes of the flow graph are given by $V_{(iii)}=U \cup W \cup (\cH \setminus S) \cup \{s,t\}$, where $s$ is a source node and $t$ is a target node. The set of edges is given by the union
\begin{align*}
    D_{(iii)} &= \{s \rightarrow w: w \in W\}  \\
    & \, \cup \{w \rightarrow h: h \in \cH \setminus S, w \in W, h \rightarrow w \in D\} \\
    & \, \cup \{h \rightarrow u: h \in \cH \setminus S, u \in U, h \rightarrow u \in D\} \\
    & \, \cup \{u \rightarrow t: u \in U\}.
\end{align*}
We assign to the source node $s$ and the target node $t$ the capacity $\infty$, while all other nodes have capacity $1$. The edges all have capacity $\infty$. By construction, no flow in $G_{\flow}^{(iii)}(v,W,U,S)$ can exceed $|W|=|U|$ in size, hence one may replace the infinite capacities with $|W|=|U|$ in practice. To check whether Condition (iv) holds, we construct a second flow graph $G_{\flow}^{(iv)}(v,W,U,S)=(V_{(iv)},D_{(iv)})$ by adding some nodes and edges to the flow graph $G_{\flow}^{(iii)}(v,W,U,S)$. Let $v'$ be a copy of $v$. Then the graph $G_{\flow}^{(iv)}(v,W,U,S)$ contains the nodes $V_{(iv)} = V_{(iii)} \cup \{v, v'\}$ and the edges
\begin{align*}
    D_{(iv)} = D_{(iii)} & \, \cup \{s \rightarrow v\} \\
    & \, \cup \{v \rightarrow h: h \in \cH \setminus S,  h \rightarrow v \in D \} \\
    & \, \cup \{h \rightarrow v': h \in \cH \setminus S, h \rightarrow v' \in D\} \\
    & \, \cup \{v' \rightarrow t \}.
\end{align*}
Similarly as before, we assign to all edges and to the source node $s$ and the target node $t$ the capacity $\infty$, while we assign to all other nodes the capacity $1$. Since no flow in $G_{\flow}^{(iv)}(v,W,U,S)$ can exceed $|W|+1$ in size, we replace the infinite capacities with $|W|+1$ in practice. An example of a both flow graphs is shown in Figure~\ref{fig:max-flow}.

\tikzset{
      every node/.style={circle, inner sep=0.3mm, minimum size=0.5cm, draw, thick, black, fill=white, text=black},
      every path/.style={thick}
}
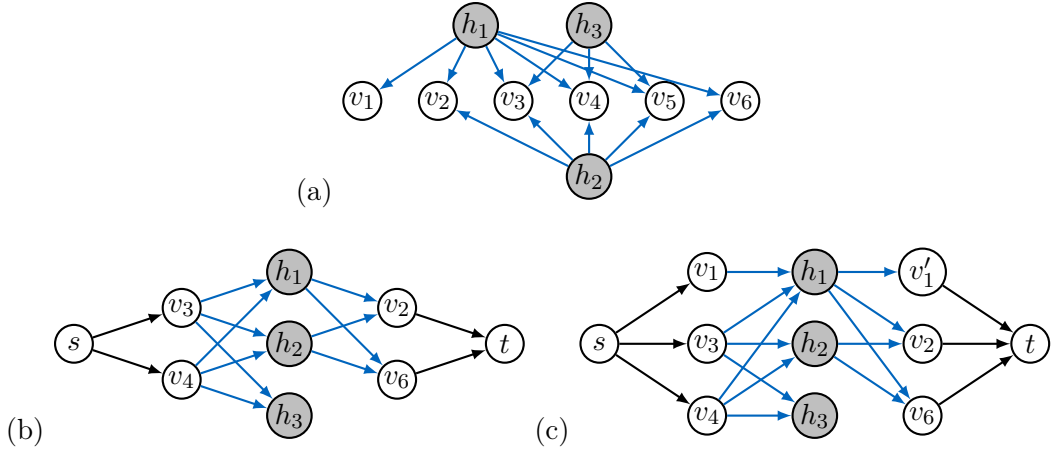
\begin{figure}[t]
\begin{center}
(a)
\begin{tikzpicture}[align=center]
    \node[fill=lightgray] (h1) at (-1,1) {$h_1$};
    \node[fill=lightgray] (h2) at (0.5,-1) {$h_2$};
    \node[fill=lightgray] (h3) at (0.5,1) {$h_3$};
    
    \node[] (1) at (-2.5,0) {$v_1$};
    \node[] (2) at (-1.5,0) {$v_2$};
    \node[] (3) at (-0.5,0) {$v_3$};
    \node[] (4) at (0.5,0) {$v_4$};
    \node[] (5) at (1.5,0) {$v_5$};
    \node[] (6) at (2.5,0) {$v_6$};
    
    \draw[MyBlue] [-latex] (h1) edge (1);
    \draw[MyBlue] [-latex] (h1) edge (2);
    \draw[MyBlue] [-latex] (h1) edge (3);
    \draw[MyBlue] [-latex] (h1) edge (4);
    \draw[MyBlue] [-latex] (h1) edge (5);
    \draw[MyBlue] [-latex] (h1) edge (6);
    \draw[MyBlue] [-latex] (h2) edge (2);
    \draw[MyBlue] [-latex] (h2) edge (3);
    \draw[MyBlue] [-latex] (h2) edge (4);
    \draw[MyBlue] [-latex] (h2) edge (5);
    \draw[MyBlue] [-latex] (h2) edge (6);
    \draw[MyBlue] [-latex] (h3) edge (3);
    \draw[MyBlue] [-latex] (h3) edge (4);
    \draw[MyBlue] [-latex] (h3) edge (5);
\end{tikzpicture} \\[0.5cm]
(b)
\begin{tikzpicture}[align=center, scale=0.95]
    \node[fill=lightgray] (h1) at (0,1) {$h_1$};
    \node[fill=lightgray] (h2) at (0,0) {$h_2$};
    \node[fill=lightgray] (h3) at (0,-1) {$h_3$};
    
    \node[] (s) at (-3,0) {$s$};
    \node[] (t) at (3,0) {$t$};
    \node[] (3) at (-1.5,0.5) {$v_3$};
    \node[] (4) at (-1.5,-0.5) {$v_4$};
    \node[] (2) at (1.5,0.5) {$v_2$};
    \node[] (6) at (1.5,-0.5) {$v_6$};    
    
    \draw[MyBlue] [-latex] (3) edge (h1);
    \draw[MyBlue] [-latex] (3) edge (h2);
    \draw[MyBlue] [-latex] (3) edge (h3);
    \draw[MyBlue] [-latex] (4) edge (h1);
    \draw[MyBlue] [-latex] (4) edge (h2);
    \draw[MyBlue] [-latex] (4) edge (h3);
    \draw[MyBlue] [-latex] (h1) edge (2);
    \draw[MyBlue] [-latex] (h2) edge (2);
    \draw[MyBlue] [-latex] (h1) edge (6);
    \draw[MyBlue] [-latex] (h2) edge (6);
    \draw[] [-latex] (s) edge (3);
    \draw[] [-latex] (s) edge (4);
    \draw[] [-latex] (2) edge (t);
    \draw[] [-latex] (6) edge (t);
\end{tikzpicture}
(c)
\begin{tikzpicture}[align=center, scale=0.95]
    \node[fill=lightgray] (h1) at (0,1) {$h_1$};
    \node[fill=lightgray] (h2) at (0,0) {$h_2$};
    \node[fill=lightgray] (h3) at (0,-1) {$h_3$};
    
    \node[] (s) at (-3,0) {$s$};
    \node[] (t) at (3,0) {$t$};
    \node[] (1) at (-1.5,1) {$v_1$};
    \node[] (3) at (-1.5,0) {$v_3$};
    \node[] (4) at (-1.5,-1) {$v_4$};
    \node[] (2) at (1.5,0) {$v_2$};
    \node[] (6) at (1.5,-1) {$v_6$};  
    \node[] (1c) at (1.5,1) {$v_1'$};
    
    \draw[MyBlue] [-latex] (1) edge (h1);
    \draw[MyBlue] [-latex] (3) edge (h1);
    \draw[MyBlue] [-latex] (3) edge (h2);
    \draw[MyBlue] [-latex] (3) edge (h3);
    \draw[MyBlue] [-latex] (4) edge (h1);
    \draw[MyBlue] [-latex] (4) edge (h2);
    \draw[MyBlue] [-latex] (4) edge (h3);
    \draw[MyBlue] [-latex] (h1) edge (1c);
    \draw[MyBlue] [-latex] (h1) edge (2);
    \draw[MyBlue] [-latex] (h2) edge (2);
    \draw[MyBlue] [-latex] (h1) edge (6);
    \draw[MyBlue] [-latex] (h2) edge (6);
    \draw[] [-latex] (s) edge (1);
    \draw[] [-latex] (s) edge (3);
    \draw[] [-latex] (s) edge (4);
    \draw[] [-latex] (1c) edge (t);
    \draw[] [-latex] (2) edge (t);
    \draw[] [-latex] (6) edge (t);
\end{tikzpicture}
\caption{Using maximum flow to verify whether the tuple $(v, W, U, S)=(v_1,\{v_3,v_4\},\{v_2,v_6\},\emptyset)$ satisfies Conditions (iii) and (iv) of the matching criterion. (a) The considered sparse factor analysis graph. (b) The flow graph $G_{\flow}^{(iii)}(v,W,U,S)$. (c) The flow graph $G_{\flow}^{(iv)}(v,W,U,S)$.}
\label{fig:max-flow}
\end{center}
\end{figure}

Let \texttt{MaxFlow}($G_{\flow}^{(iii)}(v,W,U,S)$) be the maximum flow from $s$ to $t$ in the flow graph $G_{\flow}^{(iii)}(v,W,U,S)$ and let \texttt{MaxFlow}($G_{\flow}^{(iv)}(v,W,U,S)$) be the maximum flow from $s$ to $t$ in the flow graph $G_{\flow}^{(iv)}(v,W,U,S)$. We have the following result.

\begin{lemma} \label{lem:verify-tuple}
Let $G=(V \cup \mathcal{H}, D)$ be a factor analysis graph and fix a latent node $h \in \cH$.  Suppose that the tuple $(v, W, U, S) \in V \times 2^{V} \times 2^{V} \times 2^{ \cH \setminus \{h\}}$ satisfies Conditions (i) and (ii) of the matching criterion with respect to $h$. Then, Conditions (iii) and (iv) of the matching criterion are satisfied if and only if $\mathtt{MaxFlow}(G_{\flow}^{(iii)}(v,W,U,S))=|W|$ and $\mathtt{MaxFlow}(G_{\flow}^{(iv)}(v,W,U,S))=|W|$.
\end{lemma}
\begin{proof}
The proof is along the lines of the proofs of  \citet[Theorem 6]{foygel2012halftrek} and   \citet[Theorem 5.1]{barber2022halftrek}. Let $(v, W, U, S) \in V \times 2^{V} \times 2^{V} \times 2^{ \cH \setminus \{h\}}$ be a tuple that satisfies Conditions (i) and (ii) of the matching criterion with respect to $h$. First, we show that Condition (iii) is satisfied if and only if $\mathtt{MaxFlow}(G_{\flow}^{(iii)}(v,W,U,S))=|W|$.
Suppose Condition (iii) is satisfied, i.e., there an intersection-free matching $\Pi$ of $W$ and $U$ that avoids $S$. For each path $\pi \in \Pi$ of the form
\[
    \pi: w \leftarrow h \rightarrow u,
\]
with $w \in W$, $u \in U$ and $h \in \cH \setminus S$, add a flow of size $1$ along the path
\[
    \widetilde{\pi}: s \rightarrow u \rightarrow h \rightarrow w \rightarrow t
\]
in the flow graph $G_{\flow}^{(iii)}$. Let $\widetilde{\Pi}$ be the set of paths that we obtain in the flow graph $G_{\flow}^{(iii)}$. Observe that the total flow size from $s$ to $t$ in the flow graph is equal to $|W|=|U|$. It is left to check that no capacity constraint is exceeded. This is trivial for the infinite edge capacities as well as for the infinite capacities of the source node $s$ and the target node $t$. Note that all other nodes that appear in some of the paths of the set $\widetilde{\Pi}$ appear exactly once since the original set of paths $\Pi$ is intersection-free and $W \cap U = \emptyset$. 

Now, suppose for the other direction that $\mathtt{MaxFlow}(G_{\flow}^{(iii)}(v,W,U,S))=|W|$. Based on the properties of the max-flow problem with integer capacities \citep{ford1962flows}, this implies that there are $|W|$ directed path from $s$ to $t$, each having a flow of size $1$. Let $\widetilde{\Pi}$ be the collection of these paths. Each node different from $s$ and $t$ can appear at most once in the set of paths $\widetilde{\Pi}$. By construction, each path $\widetilde{\pi} \in \widetilde{\Pi}$ has the form 
\[
    \widetilde{\pi}: s \rightarrow w \rightarrow h \rightarrow u \rightarrow t,
\]
with $w \in W$, $u \in U$ and $h \in \cH \setminus S$. This defines the set of paths $\Pi$ in the factor analysis graph $G$, where each path $\pi \in \Pi$ is of the form 
\[
    \pi: w \leftarrow h \rightarrow u.
\]
By construction, $\Pi$ is an intersection-free matching from $W$ to $U$ since each node other than $s$ or $t$ appears at most once in the set of paths $\widetilde{\Pi}$.

Equivalently, we can see that there is an intersection-free matching of $W \cup \{v\}$ and $U \cup \{v\}$ that avoids $S$ if and only if $\mathtt{MaxFlow}(G_{\flow}^{(iv)}(v,W,U,S))=|W|+1$. We conclude the proof by noting that it must hold that $\mathtt{MaxFlow}(G_{\flow}^{(iv)}(v,W,U,S))\in \{|W|, |W|+1\}$ whenever $\mathtt{MaxFlow}(G_{\flow}^{(iii)}(v,W,U,S))=|W|$. 
\end{proof}

Lemma~\ref{lem:verify-tuple} implies that the decision problem \texttt{M}($G,h,S$) is in the NP-complexity class. Every candidate tuple $(v,W,U)$ to solve \texttt{M}($G,h,S$) can be checked to be a solution in polynomial time by first checking if the tuple satisfies Conditions (i) and (ii) of the matching criterion and then verifying whether $\mathtt{MaxFlow}(G_{\flow}^{(iii)})=|W|$ and $\mathtt{MaxFlow}(G_{\flow}^{(iv)})=|W|$.

We now give an algorithm to decide \texttt{M}($G,h,S$) by iterating over all suitable tuples $(v, W, U) \in V \times 2^{V} \times 2^{V}$. By applying Lemma~\ref{lem:verify-tuple}, we check for each tuple whether it satisfies the matching criterion with respect to $h$. The next fact simplifies the search for tuples that satisfy the matching criterion.

\begin{lemma} \label{lem:choice-of-W-and-U}
    Let $G=(V \cup \mathcal{H}, D)$ be a factor analysis graph. If the tuple $(v, W, U, S) \in V \times 2^{V} \times 2^{V} \times 2^{ \cH \setminus \{h\}}$ satisfies the matching criterion with respect to a latent node $h \in \cH$, then it holds that $h \in \pa(W) \cap \pa(U)$. 
\end{lemma}
\begin{proof}
    Let $h \in \cH$ be a latent node and suppose that  the tuple $(v, W, U, S) \in V \times 2^{V} \times 2^{V} \times 2^{ \cH \setminus \{h\}}$ satisfies the matching criterion with respect to $h$. By Condition (iii) of the matching criterion, there exists an intersection-free matching $\Pi$ of $W$ and $U$ that avoids $S$. If $h \not\in \pa(W)$ or $h \not\in \pa(U)$, then $h$ does not appear in $\Pi$. Therefore, by adding the path $v \leftarrow h \rightarrow v$ to $\Pi$, we  obtain an intersection-free matching of $\{v\} \cup W$ and  $\{v\} \cup U$ that avoids $S$. This is a contradiction to Condition (iv) of the matching criterion and we conclude that it must hold that $h \in \pa(W) \cap \pa(U)$.
\end{proof}

Our procedure to decide \texttt{M}($G,h,S$) is formalized in Algorithm~\ref{alg:check-M}.

\begin{algorithm}[t]
\caption{Deciding \texttt{M}($G,h,S$).}
\begin{algorithmic}[1]
\REQUIRE Factor analysis graph $G = (V \cup \cH, D)$, a node $h \in \cH$, and a set $S \subseteq \cH \setminus \{h\}$.\\
\FOR {$v \in \ch(h)$ such that $\pa(v)\setminus S  = \{h\}$}
    \FOR {$W \subseteq V \setminus \{v\}$ such that $h \in \pa(W)$ and $|W| \leq \min\{(|V|-1)/2, |\cH \setminus S|\}$} 
        \FOR {$U \subseteq \ch(\pa(W)\setminus S) \setminus (\{v\} \cup W)$ such that $h \in \pa(U)$ and $|U|=|W|$}
            \IF {$\mathtt{MaxFlow}(G_{\flow}^{(iii)}(v,W,U,S))=|W|$ and $\mathtt{MaxFlow}(G_{\flow}^{(iv)}(v,W,U,S))=|W|$}
                \RETURN ``yes''.
            \ENDIF
        \ENDFOR
    \ENDFOR
\ENDFOR
\RETURN ``no''.
\end{algorithmic}
\label{alg:check-M}
\end{algorithm}

\begin{theorem}\label{thm:algorithm-one-node}
Algorithm \ref{alg:check-M} is sound and complete for deciding \texttt{M}($G,h,S$). If we only allow sets $W$ with $|W| \leq k$ in line $2$, then the algorithm has complexity at most $\mathcal{O}(|V|^{k+1} (|V|+|\cH|)^3)$.
\end{theorem}
\begin{proof}
The algorithm is sound and complete by Lemma~\ref{lem:choice-of-W-and-U} and Lemma~\ref{lem:verify-tuple}. For the complexity, note that we run the ``inner'' algorithm (lines 2 to 8) at most $|V|$ times. In the inner algorithm we iterate at most over all sets $W \subseteq V$ with cardinality at most $k$ in line $s$. The number of subsets of $W$ with cardinality at most $k$ is
\[
    \sum_{i=0}^k \binom{|V|}{i} = \cO(|V|^k)
\]
In line $3$ we then iterate at most over all $U \subseteq V$ with $|U|=|W|=k$. Thus, we compute at most $\cO(|V|^{k+1})$ maximum flows on acyclic graphs with at most $|V| + |\cH| + 3$ nodes. By \citet[Section 26]{cormen2009introduction}, each maximum flow computation is of complexity $\cO((|V|+|\cH|)^3)$, and we conclude that the total complexity of Algorithm~\ref{alg:check-M} is $\mathcal{O}(|V|^{k+1} (|V|+|\cH|)^3)$.
\end{proof}

Finally, we provide a procedure for deciding M-identifiability in  Algorithm~\ref{alg:check-M-id}, where we iterate over all nodes $h \in \cH \setminus S$ and solve \texttt{M}($G,h,S$) in each step. It is easy to see that the algorithm is sound and complete.  Under the same constraints as in Theorem~\ref{thm:algorithm-one-node}, the complexity is at most $\mathcal{O}(|\cH|^2|V|^{k+1} (|V|+|\cH|)^3)$.

\begin{algorithm}[t]
\caption{Deciding M-identifiability.}
\begin{algorithmic}[1]
\REQUIRE Factor analysis graph $G = (V \cup \cH, D)$.\\
\ENSURE Solved nodes $S \leftarrow \{h \in \cH: \ch(h) = \emptyset\}$.
\REPEAT 
    \FOR {$h \in \cH \setminus S$}
        \IF {\texttt{M}($G,h,S$) holds}
            \STATE {$S \leftarrow S \cup \{h\}$.}
            \BREAK
        \ENDIF
    \ENDFOR
\UNTIL{$S = \cH$ or no change has occurred in the last iteration.}
\RETURN ``yes'' if $S=\cH$, ``no'' otherwise.
\end{algorithmic}
\label{alg:check-M-id}
\end{algorithm}

\subsection{Deciding Extended M-identifiability}
In this section, we provide an algorithm for deciding extended M-identifiability in a factor analysis graph. We start by proposing a procedure to find a tuple that satisfies the local BB-criterion; recall Theorem~\ref{thm:f-identifiability}. Denote by $\texttt{L}(G,S)$ the subproblem of deciding whether there exists a set $B \subseteq V$ such that the tuple $(B,S)$ satisfies the local BB-criterion for a fixed set of latent nodes $S \subseteq \cH$. The next lemma verifies that we can always fix the set $S$ to be the set of all nodes already proven to be generically sign-identifiable.

\begin{lemma} \label{lem:F-solved-nodes}
Let $G=(V \cup \cH, D)$ be a factor analysis graph and suppose that the tuple $(B,S)\in 2^V \times 2^{\cH}$  satisfies the local BB-criterion. Let $h \in \jpa(B)\setminus S$ be a latent node. Then there is $\widetilde{B} \in 2^V$ with $|\widetilde{B}|=|B|-1$ and $\jpa(\widetilde{B}) \setminus (S \cup h) = \jpa(B) \setminus (S \cup h)$ such that $(\widetilde{B}, S \cup h)$ satisfies the local BB-criterion. 
\end{lemma}
\begin{proof}
Let $p=|B|$ and $m=|\jpa(B)\setminus S|$. Since $G[B \cup (\jpa(B)\setminus S)]$ is a full-ZUTA graph,  there is a unique observed node $u_\ell$ for each $\ell \in \jpa(B)\setminus S$ such that $u_\ell \in \ch(\ell)$ and $u_\ell \not\in \bigcup_{k \succ_{\text{ZUTA}} \ell} \ch(k)$.
Now, let $h \in \jpa(B)\setminus S$ be a latent node and define $\widetilde{B}$ as 
\[
    \widetilde{B} 
    = B \setminus \{u_h\}.
\]
Note that $|\widetilde{B}| = |B| - 1$. We claim that the pair $(\widetilde{B}, S \cup h)$ satisfies the local BB-criterion. However, we first show that   $\jpa(\widetilde{B}) \setminus (S \cup h) = \jpa(B) \setminus (S \cup h)$. Since $\widetilde{B} \subseteq B$, one inclusion is clear. For the other inclusion, take a latent node $\ell \in \jpa(B) \setminus (S \cup h)$. Observe that the corresponding node $u_\ell$ is an element of $\widetilde{B}$. Moreover, since $m < p$ and $(B,S)$ satisfies the local BB-criterion, there is $w \in B$ such that $w \in \ch(k)\setminus \{u_k\}$ for all $k \in \jpa(B)\setminus S$.  Hence, it must be the case that  $w \in \widetilde{B}$, and since $\ell \in \jpa(u_\ell,w)$, we conclude that $\ell \in \jpa(\widetilde{B})\setminus (S \cup h)$.

It is easy to see that the pair $(\widetilde{B}, S \cup h)$ satisfies Condition (i) of the local BB-criterion. To check Condition (ii), consider a latent node $\ell \in \jpa(B) \setminus (S \cup h)$. If $\ell \succ_{\text{ZUTA}} h$, then $u_h \not\in \ch(\ell)$, and the $B$-first ordering $\prec_\ell$ on $\ch(\ell)$ is also a $\widetilde{B}$-first ordering on $\ch(\ell)$ that satisfies condition (ii).

Now, consider a latent node $\ell \in \jpa(B) \setminus (S \cup h)$ with $\ell \prec_{\text{ZUTA}} h$. Note that $\prec_\ell$ is a $\widetilde{B}$-first-ordering on $\ch(\ell) \setminus \{u_h\}$. Now, we extend this ordering to an ordering $\widetilde{\prec}_\ell$ on the whole set of children $\ch(\ell)$. We define it as the block-ordering
\begin{equation} \label{eq:block-ordering}
    \widetilde{B} \,\,\,\, \widetilde{\prec}_\ell  \,\,\,\, \{u_h\} \,\,\,\, \widetilde{\prec}_\ell  \,\,\,\, \ch(\ell)\setminus B,
\end{equation}
where, for two sets $A,C$, we write $A \, \widetilde{\prec}_\ell \, C$ if $a \, \widetilde{\prec}_\ell \, c$ whenever $a \in A$ and $c \in C$ in the ordering $\widetilde{\prec}_\ell$. Moreover, within each set in~\eqref{eq:block-ordering}, the ordering $\widetilde{\prec}_\ell$ coincides with the ordering $\prec_\ell$. Clearly, the ordering in~\eqref{eq:block-ordering} is a  $\widetilde{B}$-first ordering on the set of children $\ch(h)$. To show that is satisfies condition (ii) of the local BB-criterion,  consider first the node $u_h$. Take the node $u_\ell$ and observe that it satisfies $u_{\ell} \, \widetilde{\prec}_\ell \, u_h$ and $\jpa(\{u_\ell, u_h\})\setminus (S \cup h) \subseteq \{k \in \jpa(\widetilde{B})\setminus(S \cup h): k \preceq_{\text{ZUTA}} \ell\}$. Now, take any other node $v \in \ch(\ell)\setminus B$. Then the existence of a suitable node $u \in \ch(\ell)$ is ensured by the fact that $\widetilde{B}=B \cup \{u_h\}$. Hence $(\widetilde{B}, S \cup h)$ satisfies Condition (ii) the local BB-criterion.

To finish the proof, it remains to show Condition (iii) of the local BB-criterion. Recall that, for $p=|B|$ and $m=|\jpa(B)\setminus S|$, it holds that $ p(m+1) - \binom{m}{2} < \binom{p+1}{2}$. Now, let 
$\widetilde{p}=p-1$ and $\widetilde{m}=m-1$, and consider the following chain of equivalent statements:
\begin{align*}
    &p(m+1) - \binom{m}{2} < \binom{p+1}{2} \\
    \iff &pm - m + m + p - \left\{\binom{m-1}{2}+(m-1)\right\} < \binom{p}{2} + p \\
    \iff &(p-1)m - \binom{m-1}{2} + 1 < \binom{p}{2} \\
    \iff & \widetilde{p}(\widetilde{m}+1) - \binom{\widetilde{m}}{2} < \binom{\widetilde{p}+1}{2}.
\end{align*}
We conclude that the pair $(\widetilde{B}, S \cup h)$ satisfies the local BB-criterion. 
\end{proof}

Next, we provide an algorithm for deciding whether a factor analysis graph is a full-ZUTA graph.

\begin{algorithm}[t]
\caption{Deciding full-ZUTA.}
\begin{algorithmic}[1]
\REQUIRE Factor analysis graph $G = (V \cup \cH, D)$ with $|\cH| \leq |V|$.\\
\ENSURE $p \leftarrow |V|$ and $m \leftarrow |\cH|$.
\STATE {Relabel $\cH \leftarrow \{h_1, \ldots, h_m\}$ such that $\ch(h_i) \leq \ch(h_{i+1})$ for all $i\in[m-1]$.}
\FOR {$i \in [m]$}
    \IF {$|\ch(h_i)| \neq p-i+1$}
        \STATE {\textbf{return} ``no''.}
    \ENDIF 
\ENDFOR
\FOR {$i \in [m-1]$}
    \STATE{$W \leftarrow \ch(h_i) \setminus \ch(h_{i+1})$.}
    \IF {$|W| \neq 1$ }
        \STATE {\textbf{return} ``no''.}
    \ELSE
        \STATE{ $W = \{w\}$.}
        \FOR {$j \in \{i+2, \ldots, m\}$}
            \IF {$w \in \ch(h_j)$}
                \STATE{\textbf{return} ``no''.}
            \ENDIF
        \ENDFOR
    \ENDIF
\ENDFOR
\STATE{\textbf{return} ``yes''.}
\end{algorithmic}
\label{alg:check-full-ZUTA}
\end{algorithm}

\begin{lemma} \label{lem:full-ZUTA-algo}
    A sparse factor analysis graph $G = (V \cup \cH, D)$ is a full-ZUTA graph if and only if Algorithm~\ref{alg:check-full-ZUTA} returns ``yes''. Moreover, the algorithm has complexity at most $\mathcal{O}(|\cH| |V|^2)$.
\end{lemma}
\begin{proof}
We start by analyzing the complexity of Algorithm~\ref{alg:check-full-ZUTA}. We assume w.l.o.g.~that $|V|\geq|\cH|$, since otherwise a factor analysis graph is trivially not a full-ZUTA graph. Computing the children of all latent nodes is of complexity at most $\mathcal{O}(|\cH||V|)$, and ordering the latent nodes is of complexity at most $\mathcal{O}(|\cH|^2)$. Hence, line $1$ is of complexity at most $\mathcal{O}(|\cH||V|)$. Lines $2$-$6$ are also of complexity at most $\mathcal{O}(|\cH||V|)$.
The remaining algorithm consists of two nested for-loops, both iterating over the latent nodes. Computing the set difference in line $8$ is of complexity $\mathcal{O}(|V|^2)$ and
verifying membership in the set of children in line $14$ is of complexity $\mathcal{O}(|V|)$. By considering the nested structure of the computations and using that $|\cH| \leq |V|$, we conclude that the complexity of Algorithm~\ref{alg:check-full-ZUTA} is at most $\mathcal{O}(|\cH| |V|^2)$.

Now, observe that the graph $G$ is a full-ZUTA graph if and only if there is a relabeling of the latent nodes $\cH = \{h_1, \ldots, h_m\}$ such that (i) $|\ch(h_i)|=p-i+1$, and (ii), for all latent nodes $i=1, \ldots, m-1$, there is an observed node $v_i \in \ch(h_i)$ such that $\ch(h_i) = \ch(h_{i+1}) \cup \{v_i\}$. Condition (ii) holds if and only if, for all $i=1, \ldots, m-1$, the set $W_i = \ch(h_i) \setminus \ch(h_{i+1})$ has cardinality $|W_i|=1$ and the single element  of $W_i$ is not contained in $\ch(h_j)$ for $j=\{i+2, \ldots, m\}$. We conclude that the output of Algorithm~\ref{alg:check-full-ZUTA} is ``yes'' if and only if $G$ is a full-ZUTA graph.
\end{proof}

Using Algorithm~\ref{alg:check-full-ZUTA} we can check Condition (i) of the local BB-criterion. The purpose of the next algorithm is to check Condition (ii).

\begin{algorithm}[t]
\caption{Verifying Condition (ii) of the local BB-criterion.}
\begin{algorithmic}[1]
\REQUIRE Factor analysis graph $G = (V \cup \cH, D)$ and a tuple $(B,S) \in 2^V \times 2^{\cH}$ such that  $G[B \cup (\jpa(B)\setminus S)]$ is a full-ZUTA graph with $ \prec_{\text{ZUTA}}$ being the unique ZUTA-ordering on $\jpa(B)\setminus S$.\\
\FOR {$h \in \jpa(B) \setminus S$}
\STATE \textbf{Initialize} $W \leftarrow \ch(h) \setminus B$, $B' \leftarrow B \cap \ch(h)$ and $L \leftarrow \{\ell \in \jpa(B) \setminus S : \ell \preceq_{\text{ZUTA}} h\}$.
\REPEAT 
\FOR {$v \in W$}
    \FOR {$u \in B'$}
        \IF {$\jpa(\{u,v\}) \setminus S \subseteq L$}
            \STATE {$W \leftarrow W \setminus \{v\}$ and $B' = B' \cup \{v\}$.}
            \BREAK { the two inner for-loops.}
        \ENDIF
    \ENDFOR
\ENDFOR
\UNTIL {$W=\emptyset$ or no change has occurred in the last iteration.}
\IF {no change has occurred in the last iteration and $W \neq \emptyset$}
    \STATE {\textbf{return} ``no''.}
\ENDIF
\ENDFOR
\STATE {\textbf{return} ``yes''.}
\end{algorithmic}
\label{alg:check-cond-ii}
\end{algorithm}

\begin{lemma} \label{lem:condition-ii-algo}
    Let  $G = (V \cup \cH, D)$ be a factor analysis graph and consider a tuple $(B,S) \in 2^V \times 2^{\cH}$ such that  $G[B \cup (\jpa(B)\setminus S)]$ is a full-ZUTA graph. Then, the tuple $(B,S)$ satisfies Condition (ii) of the local BB-criterion if and only if Algorithm~\ref{alg:check-cond-ii} returns ``yes''. Moreover, the algorithm has complexity at most $\mathcal{O}(|\cH|^3|V|^3)$.
\end{lemma}
\begin{proof}
It follows directly from Definition~\ref{def:full-factor-criterion} that Algorithm~\ref{alg:check-cond-ii} returns ``yes'' if and only if the tuple $(B,S)$ satisfies Condition (ii) of the local BB-criterion. Hence, we only need to show the complexity. For each latent node $h \in \jpa(B) \setminus S$, the initialization in line $2$ is of less complexity then the remaining part of the algorithm in lines $3$ to $15$. Another repetition in line $3$ occurs only if a node was removed from $W$ in the previous repetition. Hence, after $|W| \leq |V|$ repetitions of line $3$ either all nodes were removed from $W$ or the repetitions were stopped before. By counting the maximal number of repetitions in the for-loops and noting that checking whether $\jpa(\{u,v\}) \setminus S$ is a subset of $L$ is of complexity at most $\mathcal{O}(|\cH|^2)$, we conclude that Algorithm~\ref{alg:check-cond-ii} has complexity at most $\mathcal{O}(|\cH|^3|V|^3)$.
\end{proof}

To solve $\texttt{L}(G,S)$ we need to iterate over subsets $B \subseteq V$ and, for each subset, we use Algorithms~\ref{alg:check-full-ZUTA} and \ref{alg:check-cond-ii} to check the local BB-criterion. To shrink the number of possible subsets, we observe that there has to be a latent node $h \in \cH \setminus S$ such that $B$ is a subset of $\ch(h)$ since otherwise Condition (i) of the local BB-criterion can never be true. Hence, it is enough to first iterate over all latent nodes $h \in \cH \setminus S$ and then iterate over all subsets $B \subseteq \ch(h)$ for solving $\texttt{L}(G,S)$. The procedure for deciding $\texttt{L}(G,S)$ is given in Algorithm~\ref{alg:check-F}.

\begin{algorithm}[t]
\caption{Deciding $\texttt{L}(G,S)$.}
\begin{algorithmic}[1]
\REQUIRE Factor analysis graph $G = (V \cup \cH, D)$, and a set $S \subseteq \cH$.\\
\FOR {$h \in \cH \setminus S$}
    \FOR {$B \subseteq \ch(h)$ such that $|B|\geq 4$}
        \IF {$(B,S)$ satisfies the local BB-criterion (Algorithms~\ref{alg:check-full-ZUTA} and \ref{alg:check-cond-ii})}
            \RETURN ``yes''.
            \BREAK { both for-loops.}
        \ENDIF
    \ENDFOR
\ENDFOR
\RETURN ``no''.
\end{algorithmic}
\label{alg:check-F}
\end{algorithm}

\begin{theorem} \label{thm:algo-F-id}
Algorithm~\ref{alg:check-F} is sound and complete for deciding $\texttt{L}(G,S)$. If we only allow sets $B$ with $|B| \leq \ell$ in line $2$,  then the algorithm has complexity at most $\mathcal{O}(|\cH|^4 |V|^{\ell+3})$.
\end{theorem}
\begin{proof}
First, we analyze the complexity of Algorithm~\ref{alg:check-F}. We run the ``inner'' algorithm (lines $2$-$7$) at most $|\cH|$ times. In the inner algorithm itself we iterate through subsets $B \subseteq \ch(h) \subseteq V$ with cardinality at most $\ell$. As we have seen in the proof of Theorem~\ref{thm:algorithm-one-node}, the number of subsets of $V$ with cardinality at most $\ell$ is $\mathcal{O}(|V|^\ell)$. Verifying whether a given tuple $(B,S)$ satisfies the local BB-criterion is of complexity at most $|\cH|^3 |V|^3$ as we show in Lemmas~\ref{lem:full-ZUTA-algo} and \ref{lem:condition-ii-algo}. Note that computing the joint parents of a set $B$ is of less complexity. Hence, we conclude that the total complexity of Algorithm~\ref{alg:check-F}  is at most $\mathcal{O}(|\cH|^4 |V|^{\ell+3})$.

For showing that Algorithm~\ref{alg:check-F} is sound and complete, it only remains  to show that $B \subseteq \ch(h)$ for some latent node $h \in \cH$ and $|B| \geq 4$ whenever a tuple $(B,S)\in 2^{V} \times 2^{\cH}$ satisfies the local BB-criterion. Hence, suppose that $(B,S)\in 2^{V} \times 2^{\cH}$ satisfies the local BB-criterion. It has to hold that $\jpa(B) \setminus S$ is nonempty, since otherwise Conditions (i) and (iii) of the local BB-criterion can not hold simultaneously. Moreover, if Condition (i) of the local BB-criterion holds and $G[B \cup (\jpa(B) \setminus S)]$ is a full-ZUTA graph, then it is easy to see that there has to be a latent node $h \in \cH$ such that $B \subseteq \ch(h)$. Finally, Condition (iii) of the local BB-criterion holds if and only if $p \geq \lfloor m + \frac{1}{2} \sqrt{8m+1} + \frac{1}{2}\rfloor + 1$ for $p=|U|$ and $m=|\jpa(B) \setminus S|$. Since $m \geq 1$, it follows that $|U| \geq 4$.
\end{proof}

We conclude this section by providing the final procedure for deciding extended M-identifiability in Algorithm~\ref{alg:check-extended-M-id}. It is easy to see that the algorithm is sound and complete.  Under the same constraints as in Theorem~\ref{thm:algorithm-one-node} and in Theorem~\ref{thm:algo-F-id}, the complexity is at most $\mathcal{O}(|\cH|^2|V|^{\max\{k,l\}+1} (|V|+|\cH|)^3)$.

\begin{algorithm}[t]
\caption{Deciding extended M-identifiability.}
\begin{algorithmic}[1]
\REQUIRE Factor analysis graph $G = (V \cup \cH, D)$.\\
\ENSURE Solved nodes $S \leftarrow \{h \in \cH: \ch(h) = \emptyset\}$.
\REPEAT 
    \FOR {$h \in \cH \setminus S$}
        \IF {\texttt{M}($G,h,S$) holds}
            \STATE {$S \leftarrow S \cup \{h\}$.}
            \BREAK
        \ENDIF
    \ENDFOR
    \IF {$\texttt{L}(G,S)$ holds with $B \subseteq V$}
        \STATE {$S \leftarrow S \cup (\jpa(B) \setminus S)$}
    \ENDIF
\UNTIL{$S = \cH$ or no change has occurred in the last iteration.}
\RETURN ``yes'' if $S=\cH$, ``no'' otherwise.
\end{algorithmic}
\label{alg:check-extended-M-id}
\end{algorithm}

\section{Proofs} \label{sec:proofs}

\subsection{Proof for Section \ref{sec:preliminaries}}

\begin{proof}[Proof of Lemma~\ref{lem:nodes-vs-graph-id}]
It is clear by Definition~\ref{def:identifiability} that generic sign-identifiability of the whole graph implies generic sign-identifiability of all nodes $h \in \cH$. For the other direction, we first note that any finite intersection of Lebesgue measure zero sets is still  a Lebesgue measure zero set. Now, let $(\widetilde{\Omega}, \widetilde{\Lambda}) \in \mathcal{F}_G(\Omega, \Lambda)$ be a generically chosen parameter pair. Since all nodes $h \in \cH$ are generically sign-identifiable, it follows that $\widetilde{\Lambda} = \Lambda \Psi$, where $\Psi$ is a $|\cH| \times |\cH|$ diagonal matrix with entries in $\{\pm 1\}$. It remains to show that $\widetilde{\Omega}$ is equal to $\Omega$. By the definition of $\tau_G$ and since $\tau_{G}(\Omega, \Lambda) = \tau_{G}(\widetilde{\Omega}, \widetilde{\Lambda})$, we have that $\Omega + \Lambda \Lambda^{\top} = \widetilde{\Omega} + \widetilde{\Lambda} \widetilde{\Lambda}^{\top}$. Since $\Psi \Psi^{\top}$ is equal to the identity matrix, it follows that 
\[
     \widetilde{\Omega} = \Omega + \Lambda \Lambda^{\top} - \widetilde{\Lambda} \widetilde{\Lambda}^{\top}  = \Omega + \Lambda \Lambda^{\top} - \Lambda \Psi \Psi^{\top} \Lambda^{\top} = \Omega.
\]
\end{proof}

\subsection{Proof for Section \ref{sec:existing-criteria}}

\begin{proof}[Proof of Theorem~\ref{thm:anderson-rubin}]
The original statement of the theorem in \citet{anderson1956statistical} is written as a pointwise condition for full factor analysis models. We show that the original statement is equivalent to the statement presented here. Consider a matrix $\Omega \in \mathbb{R}_{>0}^{|V|}$ and a matrix $\Lambda \in \mathbb{R}^{|V| \times |\cH|}$. Given the matrix $\Sigma = \Lambda \Lambda^{\top} + \Omega$,  \citet[Theorem 5.1]{anderson1956statistical} states that ``a sufficient condition for identification of $\Omega$ and $\Lambda$ up to multiplication on the right by an orthogonal matrix is that if any row of $\Lambda$ is deleted, there remain two disjoint submatrices of rank $|\cH|$''. We call this pointwise sufficient  condition the ``row-deletion'' property. Since the row-deletion property implies identification of $\Omega$ for any pair $(\Omega, \Lambda) \in \mathbb{R}_{>0}^{|V|} \times \mathbb{R}^{|V| \times |\cH|}$, this also holds if $\Lambda$ is sparse, i.e., for pairs $(\Omega, \Lambda) \in \mathbb{R}_{>0}^{|V|} \times \mathbb{R}^D$.

Now, let $G=(V \cup \mathcal{H}, D)$ be a factor analysis graph such that ZUTA is satisfied and assume that for any deleted row of $\Lambda = (\lambda_{vh}) \in \mathbb{R}^D$ there exist two disjoint submatrices that are generically of rank $|\cH|$. Hence, if $(\Omega, \Lambda) \in \Theta_G$ is generically chosen, then the row-deletion property holds and $\pi_{\diag}(\mathcal{F}_G(\Omega, \Lambda)) = \{\Omega\}$, where $\pi_{\diag}$ is defined in~\eqref{eq:pi-diag}. It follows that  the graph $G$ is generically globally identifiable and we conclude by Lemma~\ref{lem:global-id-follows-sign-id} that $G$ is generically sign-identifiable.
\end{proof}

\begin{proof}[Proof of Corollary~\ref{cor:BB-subsumes-AR}]
    Denote $m=|\cH|$ and $p=|V|$. Since $G$ is AR-identifiable, $p \geq 2m+1$ must hold, and we have that
    \[
        p (m+1) = \frac{1}{2} p (2m+2) \leq \frac{1}{2} p (p+1) = \frac{(p+1)!}{2! (p-1)!} = \binom{p+1}{2}.
    \]
    Now, it holds that $\binom{m}{2}\geq 1$  since $m \geq 2$. We conclude that $|V|+|D|=p (m+1) - \binom{m}{2} < \binom{p+1}{2}$ and thus $G$ is BB-identifiable.
\end{proof}

\subsection{Proof for Section \ref{sec:identifiability}}

\begin{proof}[Proof of Lemma~\ref{lem:determinant}]
We first introduce some notation. Consider a collection of edges $\mathbf{M}=\{h_1 \rightarrow v_1, \ldots, h_k \rightarrow v_k\} \subseteq D$, where $h_i \in \cH$ and $v_i \in V$ for all $i=1, \ldots, k$. If all the $h_i$ are distinct and all the $v_i$ are distinct, then we say that $\mathbf{M}$ is a \emph{pairing} of $S = \{h_1, \ldots, h_k\}$ and $A=\{v_1, \ldots, v_k\}$. Now, we mimic the proof of \citet[Lemma 3.2]{sullivant2010trek}. By the Cauchy-Binet determinant expansion formula, we have
\[
    \det( [\Lambda \Lambda^{\top}]_{A,B}) = \sum_{S \subseteq \cH} \det(\Lambda_{A,S}) \det(\Lambda_{B,S}),
\]
where the sum runs over subsets $S \subseteq \cH$ with $|S|=|A|=|B|$. Let $M(S,A)$ be the set of all pairings of $S$ and $A$.   By the Lindström-Gessel–Viennot lemma \citep{gessel1985binomial, lindstrom1973on}, $\det(\Lambda_{A,S})= \sum_{\mathbf M \in M(S,A)} (-1)^{\mathbf M} \lambda^{\mathbf M}$, where $(-1)^{\mathbf M}$ is the sign of the induced permutation of $\mathbf M$ and $\lambda^{\mathbf M} = \prod_{h \rightarrow v \in \mathbf{M}} \lambda_{vh}$ is the monomial  of edge coefficients. Since each summand $\det( [\Lambda^{\top}]_{S,A}) \det ([\Lambda^{\top}]_{S,B})$ consists of a sum of monomials $\lambda^{\mathbf M_{S,A}} \lambda^{\mathbf M_{S,B}}$ in different combinations of variables, the sum $\det( [\Lambda \Lambda^{\top}]_{A,B})$  vanishes if and only if $\det( \Lambda_{A,S})$ or $\det (\Lambda_{B,S})$ is zero for all $S\subseteq |\cH|$ with $|S|=|A|=|B|$. Now, it holds that $\det( \Lambda_{A,S})$ is zero if and only if there is no pairing of $S$ and $A$.  We conclude the proof by observing that the existence of  a set $S \subseteq \cH$ such that there exists a pairing of $S$ and $A$  and a pairing of $S$ and $B$ is equivalent to the existence of an intersection-free matching of $A$ and $B$; also compare to \citet[Proposition 3.4]{sullivant2010trek}.
\end{proof}

\begin{proof}[Proof of Corollary~\ref{cor:ar-equivalence}]
Let $G=(V \cup \mathcal{H}, D)$ be a factor analysis graph such that ZUTA is satisfied. The graph $G$ is AR-identifiable if and only if for any node $v \in V$, there are two disjoint sets of nodes $U,W \subseteq V \setminus \{v\}$ with $|W|=|U|=|\cH|$ such that the submatrices $\Lambda_{W,\cH}$ and $\Lambda_{U,\cH}$ are generically of rank $|\cH|$. This is equivalent to $\det(\Lambda_{W,\cH})$ and $\det(\Lambda_{U,\cH})$ not being the zero polynomials which holds if and only if $\det([\Lambda \Lambda^{\top}]_{W,U})$ is also not the zero polynomial. Finally, by Lemma~\ref{lem:determinant}, the determinant $\det([\Lambda \Lambda^{\top}]_{W,U})$ is not the zero polynomial if and only if there is an intersection-free matching between $W$ and $U$. 
\end{proof}

\begin{proof}[Proof of Theorem~\ref{thm:identifiability}]
Fix a generically chosen parameter tuple $(\Omega, \Lambda) \in \Theta_G$ and let $\Sigma=(\sigma_{ij})=\tau_G(\Omega, \Lambda)$ be its image in $\PD(|V|)$. Consider a tuple $(\widetilde{\Omega}, \widetilde{\Lambda}) \in \mathcal{F}_G(\Omega, \Lambda)$ in the fiber of $(\Omega, \Lambda)$. 
Since all nodes $\ell \in S$ are generically sign-identifiable, it holds  that $\widetilde{\Lambda}_{\ch(\ell),\ell} = a_{\ell} \Lambda_{\ch(\ell),\ell}$ for some $a_{\ell} \in \{\pm 1\}$. We have to show that the vector $\widetilde{\Lambda}_{\ch(h),h}$ also coincides with $\Lambda_{\ch(h),h}$ up to sign. Define
\[
    \widehat{\sigma}_{uw} = \sigma_{uw} - \sum_{\ell \in \jpa(\{u,w\}) \cap S } \widetilde{\lambda}_{u, \ell} \widetilde{\lambda}_{w,\ell},
\]
for all $u,w \in V$. Here, $\widetilde{\lambda}_{u,\ell}$ denotes the entry of $\widetilde{\Lambda}$ that is indexed by row $u$ and column $\ell$. Since $\widetilde{\Lambda}_{\ch(\ell),\ell} = a_{\ell} \Lambda_{\ch(\ell),\ell}$, we have for $u \neq w$ that 
\begin{align}
\widehat{\sigma}_{u w} 
&=  \sum_{\ell \in \jpa(\{u,w\})} \lambda_{u,\ell} \lambda_{w,\ell} - \sum_{\ell\in \jpa(\{u,w\}) \cap S} a_{\ell}^2 \lambda_{u,\ell} \lambda_{w,\ell} \nonumber \\
&= \sum_{\ell \in \jpa(\{u,w\}) \setminus S}\lambda_{u,\ell} \lambda_{w,\ell}. \label{eq:sigma-hat}
\end{align}
Since $(\widetilde{\Omega}, \widetilde{\Lambda}) \in \mathcal{F}_G(\Omega, \Lambda)$, we also have $\Sigma = \tau_G(\widetilde{\Omega}, \widetilde{\Lambda})$. Hence, it also holds that
\begin{align}
\widehat{\sigma}_{u w} 
&=  \sum_{\ell \in \jpa(\{u,w\})} \widetilde{\lambda}_{u,\ell} \widetilde{\lambda}_{w,\ell} - \sum_{\ell\in \jpa(\{u,w\}) \cap S} \widetilde{\lambda}_{u,\ell} \widetilde{\lambda}_{w,\ell} \nonumber \\
&= \sum_{\ell \in \jpa(\{u,w\}) \setminus S}\widetilde{\lambda}_{u,\ell} \widetilde{\lambda}_{w,\ell}, \label{eq:sigma-hat-2}
\end{align}
Observe that the matrix $\widehat{\Sigma}=(\widehat{\sigma}_{uv})$ lies in the model $F(\widehat{G})$, where the augmented factor analysis graph $\widehat{G}=(V \cup \widehat{\mathcal{H}}, \widehat{D})$ is obtained from $G$ by removing the nodes $\ell \in S$ and their adjacent edges, that is, $\widehat{\mathcal{H}} = \mathcal{H} \setminus S$ and $\widehat{D}= \{h \rightarrow v \in D: h \in \widehat{\mathcal{H}}\}$. We define the following $(k+1)\times(k+1)$ matrix:
\[
    A = 
    \begin{pNiceArray}{c|c}[margin,parallelize-diags=false]
        \lambda_{v,h}^2 & \widehat{\Sigma}_{v, U} \\
        \hline
        \widehat{\Sigma}_{W, v} & \widehat{\Sigma}_{W, U}
    \end{pNiceArray}.
\]
Since $\pa(v) \setminus S = \{h\}$ and due to Equation \eqref{eq:sigma-hat} 
it holds that $A = (\widehat{\Lambda} \widehat{\Lambda}^{\top})_{\{v\} \cup W,\{v\} \cup U}$ with $\widehat{\Lambda}=(\lambda_{u \ell}) \in \mathbb{R}^{\widehat{D}}$. Now, by the definition of the matching criterion, there does not exist an intersection-free matching of $\{v\} \cup W$ and $\{v\} \cup U$ in the graph $G$ that avoids $S$. Hence, every  matching of $\{v\} \cup W$ and $\{v\} \cup U$ in the graph $\widehat{G}$ is also not intersection-free. By Lemma \ref{lem:determinant} we conclude that $\det(A) = 0$, that is, the determinant is equal to the zero polynomial. Expansion among the first row yields
\[
    0 = \det(A) =  \lambda_{v,h}^2 \det(\widehat{\Sigma}_{W, U}) - \underbrace{\sum_{i=1}^k (-1)^i \widehat{\sigma}_{v u_i} \det(\widehat{\Sigma}_{W, \{v\} \cup (U \setminus \{u_i\})} )}_{=:B},
\]
where $U = \{u_1, \ldots, u_k\}$. In order to solve this equation for $\lambda_{v,h}^2$ it has to hold that $\det(\widehat{\Sigma}_{W, U})$ is not zero. Indeed, by the definition of the matching criterion, we have that there is an intersection-free matching of $W$ and $U$. Since this matching avoids the set $S$ in the graph $G$, it is also an intersection-free matching in the graph $\widehat{G}$. In particular, $\det(\widehat{\Sigma}_{W, U})$ is not the zero polynomial. Since $\Lambda \in \mathbb{R}^D$ was generically chosen, we conclude  that $\det(\widehat{\Sigma}_{W, U}) \neq 0$ and we obtain  
\begin{equation} \label{eq:lambda-v-h}
    \lambda_{v,h}^2 = B / \det(\widehat{\Sigma}_{W, U}). 
\end{equation}
Now, define $\widetilde{A}$ equivalent as $A$ but replace $\lambda_{v,h}^2$ in the upper left corner with $\widetilde{\lambda}_{v,h}^2$. Recalling Equation~\eqref{eq:sigma-hat-2} and by repeating the same arguments, we also obtain  
\begin{equation} \label{eq:lambda-v-h-tilde}
   \widetilde{\lambda}_{v,h}^2 = B / \det(\widehat{\Sigma}_{W, U}).
\end{equation}
Taking Equations \eqref{eq:lambda-v-h} and \eqref{eq:lambda-v-h-tilde} together, it follows that $\widetilde{\lambda}_{v,h} = a_h \lambda_{v,h}$ for some $a_h \in \{\pm 1\}$. For the remaining children $u \in \ch(h) \setminus \{v\}$, recall from Equations \eqref{eq:sigma-hat} and \eqref{eq:sigma-hat-2} that $\widehat{\sigma}_{v u} = \lambda_{v,h}\lambda_{u,h} = \widetilde{\lambda}_{v,h}\widetilde{\lambda}_{u,h}$. Dividing by $\widetilde{\lambda}_{v,h}$ yields
\begin{equation*} \label{eq:id-2}
    \widetilde{\lambda}_{u,h} = \frac{\lambda_{v,h}\lambda_{u,h}}{\widetilde{\lambda}_{v,h}} = a_h \frac{\lambda_{v,h}\lambda_{u,h}}{\lambda_{v, h}} = a_h \lambda_{u,h},
\end{equation*}
which is also well-defined for generic parameter choices. We conclude that $\widetilde{\Lambda}_{\ch(h),h}=a_h\Lambda_{\ch(h),h}$, as claimed. Finally, note that that the set of points $(\Omega, \Lambda) \in \mathbb{R}^{|V|+|D|}$, where $\det(\widehat{\Sigma}_{W,U})$ or $\lambda_{v,h}$ is zero, defines a proper algebraic subset. Hence, it is a null set in $\Theta_G$; 
see e.g.~the lemma in   \citet{okamoto1973distinctness}.
\end{proof}

\begin{proof}[Proof of Corollary~\ref{cor:subsumes-AR}]
We begin by showing statement (i). Let $G=(V \cup \cH, D)$ be a factor analysis graph that satisfies ZUTA and is AR-identifiable. Let $\prec$ be a ZUTA-ordering on $\cH$ with respect to $G$. Fix a latent node $h \in \cH$ and assume that all nodes $\ell \in \cH$ with $\ell \prec h$ are generically sign-identifiable. To show that $G$ is M-identifiable, it is enough to show that there are $v \in V$ and $W, U \subseteq V$ such that the tuple $(v, W, U, S)$ with $S=\{\ell \in \cH : \ell \prec h\}$ satisfies the matching criterion with respect to $h$. 

Since $G$ satisfies ZUTA, there is an observed node $v \in \ch(h)$ such that $v \in \ch(h)$ and $v \not\in \bigcup_{\ell \succ h} \ch(\ell)$. Hence, $\pa(v) \setminus S = \{h\}$. By Corollary~\ref{cor:ar-equivalence}, there further exist two disjoint sets $\widetilde{W},\widetilde{U} \subseteq V \setminus \{v\}$ with $|\widetilde{W}|=|\widetilde{U}|=|\cH|$ such that there is an intersection-free matching between $\widetilde{W}$ and $\widetilde{U}$. Define $\widetilde{W}=\{w_1, \ldots, w_{|\cH|}\}$ and $\widetilde{U}=\{u_1, \ldots, u_{|\cH|}\}$ and let $\Pi=\{\pi_1, \ldots, \pi_{|\cH|}\}$ be an intersection-free matching of $\widetilde{W}$ and $\widetilde{U}$ such that $\pi_i$ is given by $w_i \leftarrow h_i \rightarrow u_i$. Note that the set of latent nodes $\{h_1, \ldots, h_{|\cH|}\}$ that appear in the matching $\Pi$ is equal to the set of all latent nodes $\cH$ since the matching is intersection-free. If a latent node $h_i \in S$, we remove $w_i$ from $\widetilde{W}$ and $u_i$ from $\widetilde{U}$. This defines $W$ and $U$ as subsets of $\widetilde{W}$ and $\widetilde{U}$ respectively. Note that the sets $W$ and $U$ are nonempty since $h \not\in S$.

To see that the tuple $(v,W,U,S)$ indeed satisfies the matching criterion, the only nontrivial condition that remains to be checked is condition (iv) of Definition~\ref{def:matching-criterion}. Suppose there is an intersection free matching of $\{v\} \cup \widetilde{W}$ and $\{v\} \cup \widetilde{U}$ that avoids $S$.  Adding the paths $\pi_i: w_i \leftarrow h_i \rightarrow u_i$ with $h_i \in S$ to this matching, gives an intersection-free matching of $\{v\} \cup \widetilde{W}$ and $\{v\} \cup \widetilde{U}$. On the other hand, any matching of $\{v\} \cup \widetilde{W}$ and $\{v\} \cup \widetilde{U}$ must intersect since $|\widetilde{W}|=|\widetilde{U}|=|\cH|$. We conclude by contradiction that there can not exist an intersection-free matching of $\{v\} \cup W$ and $\{v\} \cup U$ that avoids $S$.

We now show statement (ii). One direction follows from (i). For the other direction, let $G=(V \cup \cH, D)$ be a full-ZUTA graph that is M-identifiable. By Corollary \ref{cor:ar-equivalence}, we need to show that, for any $v \in V$, there exist two disjoint sets $W,U \subseteq V \setminus \{v\}$ with $|W|=|U|=|\cH|$ such that there is an intersection-free matching between $W$ and $U$. If the latter condition is satisfied for a node $v \in V$, we say  for simplicity that  $U$ and $W$ \textit{satisfy the AR-condition} for the node $v$. 

Since $G$ is M-identifiable, there is a relabeling of the latent nodes $\cH = \{h_1, \ldots, h_m\}$ and there are tuples  $(v_j, W_j, U_j, S_j) \in V \times 2^{V} \times 2^{V} \times 2^{ \cH \setminus \{h_j\}}$ that  satisfy the matching criterion with respect to $h_j$, such that  $i < j$ whenever $h_i \in S_j$. For the first latent node $h_1$, it must hold that $S_1=\emptyset$. 
Moreover, since $G$ is a full-ZUTA graph, we can relabel the elements of $U_1$ and relabel the elements of $W_1$ such that $U_1=\{u_1, \ldots, u_s\}$ and $W_1=\{w_1, \ldots, w_s\}$, and $h_{k+1} \in \jpa(w_k,u_k)$ for $k=1, \ldots, s$ and $s \in [|\cH|]$.

First, we observe that $|U_1|=|W_1|=s =|\cH|$. Assume that this is not true, i.e., $|U_1|=|W_1|=s < |\cH|$. Then there is an intersection-free matching of $\{v_1\} \cup U_1$ and $\{v_1\} \cup W_1$ given by the paths $v_1 \leftarrow h_1 \leftarrow v_1$ and $u_k \leftarrow h_{k+1} \leftarrow v_k$ for all $k=1, \ldots, s$, which is a contradiction to property (iv) of the matching criterion. We conclude that $|U_1|=|W_1|=|\cH|$ and, in  particular, the AR-condition is satisfied by $U=U_1$  and $W=W_1$ for the node $v_1$. By taking $U=U_1$ and $W=W_1$, the AR-condition is also satisfied for all $v \in V \setminus (U_1 \cup W_1 \cup \{v_1\})$.

It remains to show that there exist $U$ and $W$ such that the AR-condition is satisfied  for $v \in U_1 \cup W_1$. W.l.o.g.~we may assume that $v=u_k \in U_1$.  Take $U=(U_1 \setminus u_k) \cup \{v_1\}$ and $W=W_1$. Since $h_{k+1} \in \jpa(w_k,u_k)$, there exists an intersection-free matching of $U$ and $W$ given by the paths $v_1 \leftarrow h_1 \rightarrow w_1$, $u_{i-1} \leftarrow h_{i} \rightarrow w_{i}$ for all $i=2, \ldots, k$ and $u_{i} \leftarrow h_{i} \rightarrow w_{i}$ for all $i=k+1, \ldots, |\cH|$.
\end{proof}

\begin{proof}[Proof of Corollary~\ref{cor:zuta-satisfied}]
Suppose that $G=(V \cup \cH, D)$ is M-identifiable. Then there is a total ordering $\prec$ on the latent nodes $\cH$ such that $\ell \prec h$ whenever $\ell \in S_h$, where $(v_h, W_h, U_h, S_h) \in V \times 2^{V} \times 2^{V} \times 2^{ \cH \setminus \{h\}}$ satisfies the matching criterion with respect to $h$. It follows that $\ell \prec h$ if $\ell \in \pa(v_h) \setminus \{h\}$. Said differently, $v_h \in \ch(h)$ but $v_h \not\in \ch(\ell)$ if $h \prec \ell$. We conclude that the ordering $\prec$ is a ZUTA-ordering with respect to $G$. 
\end{proof}

\begin{proof}[Proof of Theorem~\ref{thm:f-identifiability}]
Fix a generically chosen parameter tuple $(\Omega, \Lambda) \in \Theta_G$ and let $\Sigma=(\sigma_{ij})=\tau_G(\Omega, \Lambda)$ be its image in $\PD(|V|)$. Consider a tuple $(\widetilde{\Omega}, \widetilde{\Lambda}) \in \mathcal{F}_G(\Omega, \Lambda)$ in the fiber of $(\Omega, \Lambda)$. By assumption, for all nodes $\ell \in S$, it holds  that $\widetilde{\Lambda}_{\ch(\ell),\ell} = a_{\ell} \Lambda_{\ch(\ell),\ell}$ for some $a_{\ell} \in \{\pm 1\}$. We have to show for all $h \in \jpa(B) \setminus S$   that $\widetilde{\Lambda}_{\ch(h),h}$ also coincides with $\Lambda_{\ch(h),h}$ up to sign. As in the proof of Theorem \ref{thm:identifiability}, define
\[
    \widehat{\sigma}_{vw} = \sigma_{vw} - \sum_{\ell \in \jpa(\{v,w\}) \cap S } \widetilde{\lambda}_{v, \ell} \widetilde{\lambda}_{w,\ell},
\]
for all $v,w \in V$. Since $\widetilde{\Lambda}_{\ch(\ell),\ell} = a_{\ell} \Lambda_{\ch(\ell),\ell}$, we have for $v\neq w$ that 
\begin{align*}
\widehat{\sigma}_{v w} 
&=  \sum_{\ell \in \jpa(\{v,w\})} \lambda_{v,\ell} \lambda_{w,\ell} - \sum_{\ell\in \jpa(\{v,w\}) \cap S} a_{\ell}^2 \lambda_{v,\ell} \lambda_{w,\ell} \nonumber \\
&= \sum_{\ell \in \jpa(\{v,w\}) \setminus S}\lambda_{v,\ell} \lambda_{w,\ell}. 
\end{align*}
Since $(\widetilde{\Omega}, \widetilde{\Lambda}) \in \mathcal{F}_G(\Omega, \Lambda)$, we also have $\Sigma = \tau_G(\widetilde{\Omega}, \widetilde{\Lambda})$. Hence, it also holds that
\begin{align*}
\widehat{\sigma}_{v w} 
&=  \sum_{\ell \in \jpa(\{v,w\})} \widetilde{\lambda}_{v,\ell} \widetilde{\lambda}_{w,\ell} - \sum_{\ell\in \jpa(\{v,w\}) \cap S} \widetilde{\lambda}_{v,\ell} \widetilde{\lambda}_{w,\ell} \nonumber \\
&= \sum_{\ell \in \jpa(\{v,w\}) \setminus S}\widetilde{\lambda}_{v,\ell} \widetilde{\lambda}_{w,\ell}. 
\end{align*}
Observe that the submatrix $\widehat{\Sigma}_{B,B}$ of the matrix $\widehat{\Sigma}=(\widehat{\sigma}_{vw})$ lies in the model $F(\widetilde{G})$, where $\widetilde{G} = G[B \cup (\jpa(B)\setminus S)]$. Let $\cL = \jpa(B) \setminus S$ be the latent nodes of $\widetilde{G}$, and let $\widetilde{D}$ be the edge set of $\widetilde{G}$. Since $\widetilde{G}$ is a full-ZUTA graph and $|B| + |\widetilde{D}| < \binom{|B|+1}{2}$, it follows from Theorem~\ref{thm:bekker-berge} that $\widetilde{G}$ is generically sign-identifiable. Recalling Definition~\ref{def:identifiability}, this means that 
\begin{equation} \label{eq:U-identifiable}
    \widetilde{\Lambda}_{B,\cL} = \Lambda_{B,\cL} \Psi,
\end{equation}
where $\Psi \in \mathbb{R}^{|\cL| \times |\cL|}$ is a diagonal matrix with diagonal entries in $\{\pm 1\}$. It remains to show that Equation~\eqref{eq:U-identifiable} also holds for all nodes in $V \setminus B$, i.e.,  we need to show that $\widetilde{\Lambda}_{V,\cL} = \Lambda_{V,\cL} \Psi$. We will show this by a double induction: We first induct on the  unique ZUTA-ordering $\prec_{\text{ZUTA}}$ of the latent nodes $\cL$ with respect to $\widetilde{G}$. Within each induction step, where we consider a fixed node $h \in \cL$, we then induct on the $U$-first-ordering $\prec_h$ on $\ch(h)$. 

Let $h \in \cL$ and assume that $\widetilde{\Lambda}_{V, \ell} = a_{\ell} \Lambda_{V,\ell}$ for all $\ell \prec_{\text{ZUTA}} h$, where $a_{\ell} \in \{\pm 1\}$. Moreover, consider a node $v \in \ch(h)$. For the base case of the induction, we note that $B \cap \ch(h) \neq \emptyset$ because of Conditions (i) and (iii) of the local BB-criterion. By Equation~\eqref{eq:U-identifiable}, it holds   for all $v \in \ch(h) \cap B$ that  $\widetilde{\lambda}_{v,h} = a_h \lambda_{v,h}$, where $a_h = \Psi_{hh} \in \{\pm 1\}$. 

Now, let $v \in \ch(h) \setminus B$ and assume that $\widetilde{\lambda}_{u,h} = a_h \lambda_{u,h}$ whenever $u \prec_h v$ for $u \in \ch(h)$. By property (ii) of the local BB-criterion, there has to be $u \in \ch(h)$ with $u \prec_h v$ such that $\jpa(\{v,u\}) \setminus S \subseteq \{\ell \in \cL: \ell \preceq_{\text{ZUTA}} h \}$. Define $T=S \cup \{\ell \in \cL: \ell \prec_{\text{ZUTA}} h \}$ and consider the quantity
\[
    \overline{\sigma}_{vu} = \sigma_{vu} - \sum_{\ell \in \jpa({v,u}) \cap T} \widetilde{\lambda}_{v, \ell} \widetilde{\lambda}_{u, \ell}.
\]
Observe $\jpa({v,u}) \setminus T = \{h\}$. Since $\Sigma = \tau_G(\Omega, \Lambda) = \tau_G(\widetilde{\Omega}, \widetilde{\Lambda})$, it follows similarly as above that 
\[
    \overline{\sigma}_{vu} = \lambda_{v, h} \lambda_{u, h} = \widetilde{\lambda}_{v, h} \widetilde{\lambda}_{u, h}.
\]
Dividing by $\widetilde{\lambda}_{u,h}= a_h \lambda_{u,h}$ yields
\begin{equation*}
    \widetilde{\lambda}_{v,h} = \frac{\lambda_{v,h}\lambda_{u,h}}{\widetilde{\lambda}_{u,h}} = a_h \frac{\lambda_{v,h}\lambda_{u,h}}{\lambda_{u, h}} = a_h \lambda_{v,h},
\end{equation*}
which is well-defined for generic parameter choices. We conclude that $\widetilde{\Lambda}_{\ch(h),h}=a_h\Lambda_{\ch(h),h}$.
\end{proof}

\begin{proof}[Proof of Corollary~\ref{cor:full-ZUTA-BB-iff-F}]
Let $G=(V \cup \cH, D)$ be a full-ZUTA graph that is BB-identifiable.  It is enough to show that the tuple $(B,S)=(V, \emptyset)$ satisfies the local BB-criterion.  Condition (i) is satisfied since BB-identifiability implies $|\cH| < |V|$, and hence $\jpa(V)= \cH$, which yields that the induced subgraph $\widetilde{G}=G[V \cup \jpa(V)]$ is equal to $G$. It follows that Condition (iii) is also satisfied since BB-identifiability implies $|V| + |D| < \binom{p+1}{2}$. To conclude, observe that there is nothing to show in Condition (ii) since $\ch(h)\setminus V = \emptyset$ for all $h \in \cH$.

For the other direction, let $G=(V \cup \cH, D)$ be a full-ZUTA graph and suppose that we can certify generic sign-identifiability of $G$ by recursively applying Theorem~\ref{thm:f-identifiability}. It is enough to show that whenever the local BB-criterion is satisfied in a full-ZUTA graph for some tuple $(B,S)\in 2^V \times 2^{\mathcal{H}}$ with $S=\emptyset$, then it must hold that $\jpa(B)=\cH$. Since the induced graph $G[B \cup \cH]$ is a full-ZUTA graph and Condition (iii) holds, we can then replace $B$ by $V$ and Condition (iii) still holds, i.e., the graph $G$ is BB-identifiable.

Let $\widetilde{D}$ be the edge set of the induced subgraph $G[B \cup \jpa(B)]$
To show that $\jpa(B)=\cH$, we first observe that $\jpa(B)$ can not be the empty set since in this case $|B|\leq 1$ and hence Condition (iii) does not hold. Thus it must be that $|B| > 1$ and $|\jpa(B)|\geq 1$.  It holds that $|B|+|\widetilde{D}|<\binom{|B|+1}{2}$ if and only if $|B| \geq \lfloor |\jpa(B)| + \frac{1}{2} \sqrt{8|\jpa(B)|+1} + \frac{1}{2}\rfloor + 1$. Since $|\jpa(B)| \geq 1$, this implies, in particular, that $|B| \geq |\jpa(B)|+2$. Now, suppose $\jpa(B) \neq \cH$. Since $|B| \geq |\jpa(B)|+2$ and $G$ is a full-ZUTA graph, there must be two nodes $u,w \in B$ such that there is a latent node $h \in \jpa(\{u,w\}) \setminus \jpa(B)$. This is a contradiction and we conclude that $\jpa(B) = \cH$. 
\end{proof}

\section{Deciding Identifiability by Computational Algebra} \label{sec:computational-algebra}
Generic sign-identifiability may be decided by computational algebraic geometry. We make use of the following lemma where, for a symmetric matrix $M \in \mathbb{R}^{p \times p}$, we denote by $\od(M) \in \mathbb{R}^{\binom{p}{2}}$ the vector of off-diagonal entries of $M$.

\begin{lemma} \label{lem:off-diago}
A factor analysis graph $G=(V \cup \cH, D)$ is generically sign-identifiable if and only if the map
\begin{align*}
    \phi_G: \mathbb{R}^D &\longrightarrow \mathbb{R}^{\binom{|V|}{2}} \\
    \Lambda & \longmapsto \od(\Lambda \Lambda^{\top})
\end{align*}
has fibers of the form
\[
    \phi^{-1}_G(\phi_G(\Lambda)) = \{\widetilde{\Lambda} \in \mathbb{R}^D : \widetilde{\Lambda}=\Lambda \Psi \text{ for } \Psi \in \{\pm 1\}^{|\cH| \times |\cH|} \text{ diagonal}\}
\]
for almost all $\Lambda \in \mathbb{R}^D$.
\end{lemma}
\begin{proof}
Let $G=(V \cup \cH, D)$ be a factor analysis graph. 
Take a generic tuple $(\Omega, \Lambda) \in \Theta_G$, and assume that the fiber $\phi^{-1}_G(\phi_G(\Lambda))$ is of the form as in the statement. Now, consider another tuple $(\widetilde{\Omega}, \widetilde{\Lambda}) \in \mathcal{F}_G( \Omega, \Lambda)$. Since $\Omega + \Lambda \Lambda^{\top} = \widetilde{\Omega} + \widetilde{\Lambda} \widetilde{\Lambda}^{\top}$ and $\Omega$ and $\widetilde{\Omega}$ are diagonal, we also have the equality $\od(\Lambda \Lambda^{\top}) = \od(\widetilde{\Lambda} \widetilde{\Lambda}^{\top})$. It follows that $\widetilde{\Lambda}=\Lambda \Psi$, where $\Psi$ is a $|\cH| \times |\cH|$ diagonal matrix with entries in $\{\pm 1\}$. But then we also have that
\[
     \widetilde{\Omega} = \Omega + \Lambda \Lambda^{\top} - \widetilde{\Lambda} \widetilde{\Lambda}^{\top}  = \Omega + \Lambda \Lambda^{\top} - \Lambda \Psi \Psi^{\top} \Lambda^{\top} = \Omega,
\]
where we have used that $\Psi \Psi^{\top}$ is equal to the identity matrix. Hence, we have shown that $G$ is generically sign-identifiable.

For the other direction, consider a generic matrix $\Lambda \in \mathbb{R}^D$ and assume that $G$ is generically sign-identifiable. For any matrix $\widetilde{\Lambda} \in \phi^{-1}_G(\phi_G(\Lambda))$, it holds that $\od(\Lambda \Lambda^{\top}) = \od(\widetilde{\Lambda} \widetilde{\Lambda}^{\top})$. Hence, 
$
    \Lambda \Lambda^{\top} - \widetilde{\Lambda} \widetilde{\Lambda}^{\top} = D,
$
where $D$ is a diagonal matrix. Now, consider another diagonal matrix $Q$   that is positive definite with $\min_{v \in V}Q_{vv} > - \min_{v \in V}D_{vv}$ if $\min_{v \in V}D_{vv}$ is negative. We have the equality
\[
    \Lambda \Lambda^{\top} + Q = \widetilde{\Lambda} \widetilde{\Lambda}^{\top} + D + Q,
\]
where, by construction, both $Q$ and $D + Q$ are positive definite. Note that $Q$ can be chosen from a set that has positive measure in $\mathbb{R}^{|V|}$. By the generic sign-identifiability of $G$ we can thus conclude that $D=0$ and that $\widetilde{\Lambda}=\Lambda \Psi$, where $\Psi$ is a $|\cH| \times |\cH|$ diagonal matrix with entries in $\{\pm 1\}$. 
\end{proof}

Let $\mathcal{F}_{\phi}(\Lambda)= \{\widetilde{\Lambda} \in \mathbb{R}^D: \phi_G(\widetilde{\Lambda}) = \phi_G(\Lambda)\}$ be the fiber of a matrix $\Lambda \in \mathbb{R}^D$ under the map $\phi_G$ given in Lemma~\ref{lem:off-diago}. Moreover, we denote by $\mathcal{F}_{\phi, \mathbb{C}}(\Lambda) = \{\widetilde{\Lambda} \in \mathbb{C}^D: \phi_G(\widetilde{\Lambda}) = \phi_G(\Lambda)\}$  the complex fiber under the extension of the map $\phi_G$ to the domain $\mathbb{C}^D$. In the following, we explain how to obtain the complex fiber $\mathcal{F}_{\phi, \mathbb{C}}(\Lambda_0)$ of a generically chosen parameter point $\Lambda_0 \in \mathbb{R}^D$ by computational algebra, and we discuss how this may allow us to determine generic sign-identifiability over the real numbers. \\

We denote by $\mathbb{R}[\lambda_{vh}: h \rightarrow v \in D]$ the ring over the indeterminates corresponding to the edges in the graph $G$ with real coefficients. Moreover, from now on, we let $\Lambda \in \mathbb{R}^D$ be the matrix with entries $\Lambda_{vh}$ being the indeterminates $\lambda_{vh}$ whenever $h \rightarrow v \in D$ and zero otherwise. For a randomly chosen matrix $\Lambda_0 \in \mathbb{R}^D$, we compute a reduced Gröbner basis for the equation system
\begin{equation} \label{eq:generators-ideal}
    \phi_G(\Lambda_0) - \od(\Lambda \Lambda^{\top}) 
\end{equation}
with an arbitrary term order on the indeterminates $\lambda_{vh}$; see~\citet{cox2007ideals} for background on Gröbner bases. The reduced Gröbner basis allows us to compute both fibers $\mathcal{F}_{\phi, \mathbb{C}}(\Lambda_0)$ and $\mathcal{F}_{\phi}(\Lambda_0)$. If $\Lambda_0$ is drawn from a continuous probability distribution, then the dimension and cardinality of the complex fibers $\mathcal{F}_{\phi, \mathbb{C}}(\Lambda_0)$ coincide with probability one. However, this is not true for the real fibers $\mathcal{F}_{\phi}(\Lambda_0)$.

We now explain how we determine generic-sign identifiability when knowing the complex fiber $\mathcal{F}_{\phi, \mathbb{C}}(\Lambda_0)$ of a generically chosen parameter matrix $\Lambda_0 \in \mathbb{R}^D$. 
Let $\mathbf{\Pi}:=\{\pm 1\}^{|\cH| \times |\cH|}$ be the group of diagonal matrices with diagonal entries in $\{\pm 1\}$. For $\Lambda, \widetilde{\Lambda} \in \mathbb{C}^D$, we define the equivalence relation
\[
    \Lambda \sim \widetilde{\Lambda} \text{ if there is } \Psi \in \mathbf{\Pi} \text{ such that } \widetilde{\Lambda} \Psi = \Lambda,
\]
and denote by $\mathcal{F}_{\phi, \mathbb{C}}(\Lambda_0) / \mathbf{\Pi}$ the set of equivalence classes of $\mathcal{F}_{\phi, \mathbb{C}}(\Lambda_0)$.  Note that for a generically chosen point $\Lambda_0 \in \mathbb{R}^D$, the number of complex equivalence classes $|\mathcal{F}_{\phi, \mathbb{C}}(\Lambda_0)/ \mathbf{\Pi}|$ is always the same.

\begin{definition}
Let $G=(V \cup \cH, D)$ be a factor analysis graph and let $\Lambda_0 \in \mathbb{R}^D$ be a generic parameter matrix. We say that the number of complex equivalence classes $|\mathcal{F}_{\phi, \mathbb{C}}(\Lambda_0)/ \mathbf{\Pi}| \in \mathbb{N} \cup \{\infty\}$ is the \emph{degree of sign-identifiability}.
\end{definition}

To guard against false conclusions, we repeat the randomized calculations to determine the degree of sign-identifiability several times for each graph in practice. For computing a reduced Gröbner basis one can use any computer algebra system such as SINGULAR \citep{Singular}, Macaulay2 \citep{Macaulay} or SageMath \citep{SageMath}. 
A factor analysis graph is generically sign-identifiable if its degree of sign-identifiability is $1$. If the degree of sign-identifiability is infinite, then the real fiber $\mathcal{F}_{\phi}(\Lambda)$ is also infinite \citep[Lemma 9]{whitney1957elementary}. If the degree of identifiability is a finite number larger or equal to $2$, then generic sign-identifiability may or may not hold. Formally, one needs to verify that the set of $\Lambda_0 \in \mathbb{R}^D$ where the number of real equivalence classes $|\mathcal{F}_{\phi}(\Lambda_0) / \mathbf{\Pi}|$ is larger or equal than $2$ has zero measure. In practice, we compute the number $|\mathcal{F}_{\phi}(\Lambda_0) / \mathbf{\Pi}|$ for several random draws of $\Lambda_0 \in \mathbb{R}^D$ and conclude that the graph is not generically sign-identifiable if we find $|\mathcal{F}_{\phi}(\Lambda_0) / \mathbf{\Pi}| \geq 2$ for at least two random draws of $\Lambda_0$. In our experiments in Section~\ref{sec:experiments} we always found that if the degree of identifiability is a finite number larger or equal to $2$, then the graph is not generically sign-identifiable.

\end{appendix}

\

\end{document}